\documentclass{article}

\usepackage{graphicx}

\usepackage[french,english]{babel}

\usepackage[utf8]{inputenc}
\usepackage{amsmath, amssymb, amsthm}

\usepackage{subfiles}
\usepackage{xr} 
\usepackage{hyperref}
\usepackage[capitalise]{cleveref}
\usepackage{enumitem}

\usepackage{mathtools}

\usepackage{mathrsfs}
\usepackage[scr=rsfso]{mathalpha}

\usepackage{geometry}
\newcommand{\op}[1]{\operatorname{#1}}

\newcommand{\N}{\mathbb{N}}
\newcommand{\R}{\mathbb{R}}
\newcommand{\Z}{\mathbb{Z}}

\newcommand{\W}{\mathcal{W}}

\newcommand{\lk}{\op{Lk}}
\newcommand{\lplus}{\op{+\!\!\!\!+}}

\usepackage{subcaption}
\usepackage{graphicx}

\newtheorem{theorem}{Theorem}[section]
\newtheorem{proposition}[theorem]{Proposition}

\newtheorem{corollary}[theorem]{Corollary}
\newtheorem{lemma}[theorem]{Lemma}

\theoremstyle{definition}
\newtheorem{definition}[theorem]{Definition}
\newtheorem{example}[theorem]{Example}

\usepackage{tikz}
\usetikzlibrary{shapes.geometric}
\usetikzlibrary{calc}
\usetikzlibrary{quotes, angles}
\usetikzlibrary{graphs}
\usetikzlibrary{positioning}
\usetikzlibrary{decorations.markings}

\usepackage{newpxtext}
	
\title{The Parallel Wall Theorem for CAT(0) even 2-complexes}
\author{Carl Kristof-Tessier\footnote{Supported by the FRQNT Doctoral Research Scholarship}}
\date{}

\begin{document}

\maketitle

\begin{abstract}
We prove the Parallel Wall Theorem for $\mathrm{CAT(0)}$ 2-complexes constructed by regular polygons with an even number of sides. This result extends a combination of the works of Janzen and Wise and Hruska and Wise.
\end{abstract}

\section{Introduction}

Coxeter groups are a rich class of non-positively curved groups: these groups act by reflections on an associated piecewise Euclidean complex, called the \textit{Davis complex} \cite{Davis2008GeomCoxeterGps}. 
In \cite{Moussong1988}, Moussong shows that these complexes are CAT(0).
Furthermore, the combinatorial properties of Coxeter groups are governed by \textit{walls}, i.e.\ fixed point sets of reflections in the Davis complex. Walls are closed, convex and separate the Davis complex into two components. It follows that the combinatorial distance between vertices of the Davis complex is equal to the number of walls separating them.

As such, a major theme in understanding the geometry of Coxeter groups comes down to the study of their underlying wallspaces. Many of the results in this direction hinge on the Parallel Wall (PW) Theorem, which  states that walls that are not separated from a vertex by another wall of the Davis complex stay at a uniformly bounded distance from said vertex. 

The PW Theorem was first proven in \cite{BrinkHowlett1993}, in which the result is used to provide an automatic structure on Coxeter groups. Later, Niblo and Reeves use the PW Theorem to cubulate Coxeter groups, with the goal of showing these groups are biautomatic \cite{NibloReeves2003}. However, Coxeter groups fail in general to be cocompactly cubulated. Regardless, the PW Theorem is used to show biautomaticity for 2-dimensional Coxeter groups in \cite{Munro2021} and, finally, biautomaticity is established in full generality in \cite{OsajdaPrzytycki}. 

A Coxeter group is 2-dimensional if its Davis complex is.
The Davis complexes of 2-dimensional Coxeter groups are conceptually simple to describe; these complexes are examples of \textbf{even} 2-complexes, 2-complexes whose cells are regular polygons with an even number of sides.
In general, CAT(0) even 2-complexes have an analogous notion of walls as isometrically embedded trees, allowing us to further understand their geometry. 
Our main result generalizes the cornerstone result in Coxeter groups to this class of CAT(0) 2-complexes:

\begin{theorem}\label{thm:2D_par_wall}
    Let $X$ be a $\mathrm{CAT(0)}$ even 2-complex with $\op{Shapes(X)}$ finite. Then there is a bound $K$, depending only on $\op{Shapes}(X)$, satisfying the following: for any wall $\W$ and vertex $v\in X^{(0)}$ at combinatorial distance $\geq K$ from $\W$, there is a wall $\W'$ separating $v$ from $\W$.
\end{theorem}

Note that \cref{thm:2D_par_wall} does not require our spaces to have any symmetry whatsoever. We do not assume the complexes are cocompact, or even locally finite. We only assume that the set of isometry classes of cells, $\op{Shapes}(X)$, is finite so that the piecewise Euclidean metric is complete and geodesic. One may think of this assumption as the analogue of finite-dimensionality in cube complexes.

This extends the work of \cite{JanzenWise13} and \cite{HruskaWise14}, who indirectly show \cref{thm:2D_par_wall} under the assumption that $X$ admits a geometric group action.
Janzen and Wise show that the universal cover of a \textit{compact} non-positively curved even 2-complex $Y$ is relatively cocompactly cubulated with locally finite, finite dimensional peripheries. The work of Hruska and Wise implies that the cubulation of $\widetilde{Y}$ is locally finite which they show implies \cref{thm:2D_par_wall} for $\widetilde{Y}$.

In contrast, our proof is purely geometric, relying on metric properties of CAT(0) spaces. We use an elementary criterion for two convex subsets of a CAT(0) space to be disjoint in terms of Alexandrov angles (\cref{sec:angles}). In \cref{sec:even_2complexes}, we define the walls of CAT(0) even 2-complexes and show that they are convex. 
We also consider \textit{truncated piecewise Euclidean structures} (\cref{sec:pw_euc_structs}) in order to restrict the configurations of links in our complexes. This allows us to reduce our proof to finitely many cases, which we deal with in \cref{sec:proof}.

\vspace{1em}
\textbf{Acknowledgements} I would like to thank Piotr Przytycki for supervising me throughout and beyond this project. I would also like to thank Dani Wise for reading an initial version of this paper and offering insightful comments.

\section{Angles and convex subspaces of CAT(0) spaces}
\label{sec:angles}

We briefly review some basic facts in $\mathrm{CAT(0)}$ geometry used in the proof of \cref{thm:2D_par_wall}. We refer to \cite{BridsonHaefliger1999} for an in depth treatment of $\mathrm{CAT(0)}$ spaces.

Let $X$ be a (uniquely) geodesic space. Following \cite{BridsonHaefliger1999}, we denote by $[x, y]$ a (the) geodesic connecting $x$ to $y$. 
Given $x,y,z\in X$, the \textbf{comparison angle} $\overline{\angle_z}(x,y)$ is the Euclidean angle in the comparison triangle $\overline{\Delta}(x,y,z)$ at the vertex $\overline{z}$.
Given geodesics $\gamma_1, \gamma_2$ with $\gamma_1(0)=\gamma_2(0)$, we may then define the \textbf{Alexandrov angle} between $\gamma_1$ and $\gamma_2$ as
\[
    \angle(\gamma_1,\gamma_2)=\lim_{\epsilon\to 0} \inf_{t,t'\leq \epsilon} \overline{\angle_p}(\gamma_1(t),\gamma_2(t'))
\]
We denote by $\angle_z(x,y)$ the Alexandrov angle between segments $[z,x]$ and $[z,y]$.

\begin{proposition}[{\cite[Prop. II.1.7]{BridsonHaefliger1999}}]
    \label[proposition]{prop:cat0_angle_ineq}
    $X$ is $\mathrm{CAT(0)}$ if and only if $\angle_z(x,y)\leq \overline{\angle_z}(x,y)$ for all $x,y,z\in X$. In particular, the sum of Alexandrov angles in a geodesic triangle of a $\mathrm{CAT(0)}$ space is bounded by $\pi$.
\end{proposition}

CAT(0) spaces are always uniquely geodesic, and behave nicely with respect to their convex subsets:

\begin{proposition}[{\cite[Prop. II.2.4]{BridsonHaefliger1999}}]
    \label[proposition]{prop:projection_existence_uniqueness}
    Let $X$ be a complete $\mathrm{CAT(0)}$ space, and $C\subset X$ a closed convex subset of $X$. Then for $x\in X$, there is a unique point $\pi_C(x)\in C$ such that $d(x,\pi_C(x))=d(x,C)$. Furthermore, if $x\notin C$ then $\pi_C(x)$ is the unique point in $X$ satisfying $\angle_{\pi_C(x)}(x,y)\geq \pi/2$ for~all~$y\in~ C\setminus \{\pi_C(x)\}$.
\end{proposition}

\begin{lemma}[The disjointness criterion]
    \label[lemma]{lem:parallel_criterion}
    Let $A, B$ be closed convex subsets of a complete $\mathrm{CAT(0)}$ space $X$, and let $x\in X$ be a point such that $\pi_A(x), \pi_B(x)$ and $x$ are distinct. Suppose further that $\angle_x(\pi_A(x),\pi_B(x))=\pi$. Then $A$ and $B$ are disjoint.
\end{lemma}

\begin{proof}
    Suppose that $A$ and $B$ intersect at some point $y$. Then by \cref{prop:projection_existence_uniqueness}, we have $\angle_{\pi_A(x)}(y,x)\geq \pi/2$ and $\angle_{\pi_B(x)}(y,x)\geq \pi/2$. But our assumption $\angle_x(\pi_A(x),\pi_B(x))=\pi$ implies that the segment $[\pi_A(x),x)]\cup[x, \pi_B(x)]$ is geodesic. Thus the geodesic triangle $\Delta(y, \pi_A(x), \pi_B(x))$ has two right angles, which contradicts \cref{prop:cat0_angle_ineq}.
\end{proof}

\section{Even 2-complexes}
\label{sec:even_2complexes}

\begin{definition}
    An \textbf{even 2-complex} is a 2-complex in which every 2-cell is isometric to a regular polygon with an even number of sides.
\end{definition}

In what follows, we will always assume that $\op{Shapes}(X)$, the set of isometry classes of cells in~$X$, is finite, so that the piecewise euclidean metric on $X$ is well defined \cite{BridsonHaefliger1999}. In particular, we assume that $X$ is complete and geodesic.

\begin{definition}
    Let $X$ be a 2-complex. The \textbf{link of a vertex} $v\in X^{(0)}$, denoted by $\lk(v, X)$, is a metric graph with 
    \begin{itemize}
        \item vertices correspond to edges of $X$ containing $v$,
        \item and edges of length $\angle(e, e')$ between the vertices of $\lk(v, X)$ corresponding to edges $e, e'$ whenever $e$ and $e'$ bound a 2-cell in $X$.
    \end{itemize}
\end{definition}

In even 2-complexes, the link of a vertex $v\in X$ can be thought of as a unit sphere around $v$ endowed with the intrinsic path metric. Every geodesic $[v, x]$ starting at $v$ has a corresponding point $\vec{x}\in\lk(v, X)$, which in our case can be obtained by intersecting $[v, x]$ with the unit sphere around $v$ (after possibly extending the geodesic $[v, x]$).

\begin{lemma}
\label[lemma]{lem:angle_linkpath}
    Let $X$ be a 2-dimensional piecewise Euclidean complex (with $\op{Shapes}(X)$ finite). Then the Alexandrov angle between geodesics $[v, x]$ and $[v, y]$ starting at a vertex $v$ is the minimum between $\pi$ and the shortest path in $\op{Lk}(v,X)$ between the points $\vec{x}, \vec{y}\in\op{Lk}(v,X)$.
\end{lemma}

\begin{proposition}[\cite{BridsonHaefliger1999}]
    $X$ is $\mathrm{CAT(0)}$ if and only if $X$ is simply connected and satisfies the \textbf{link condition}: for every vertex $v\in X^{(0)}$, $\lk(v, X)$ is a simplicial graph with no embedded cycles of length~$<2\pi$.
\end{proposition}

\subsection{Walls in even 2-complexes}

Walls of an even 2-complex have been discussed in \cite{JanzenWise13} as walls in its rhombic subdivision. Here, we consider a slightly different construction that agrees with Janzen and Wise's walls on the 1-skeleton. However, our walls are also convex in the piecewise Euclidean metric when the complex is CAT(0).

Let $X$ be an even 2-complex, and let $P$ be an even polygon in $X$. Then every edge of $P$ has an opposite edge. The \textbf{mirror} dual to a pair of opposite edges $e, e'\subset P$, denoted by $M(e, e')$, is the segment in $P$ connecting the midpoint of $e$ with the midpoint of $e'$.  
We say that two edges of an even 2-complex $X$ are equivalent if they are joint by a sequence of opposite edges in $X$. The equivalence class of an edge $e\subset X$ is denoted by $[e]$.

\begin{definition}
    The \textbf{wall dual to the edge} $e\subset X$ is the abstract complex obtained by taking mirrors dual to opposite edges in $[e]$ and identifying endpoints of two mirrors when the endpoints are the midpoint of the same edge in $X$.
\end{definition}

\begin{proposition}
\label[proposition]{prop:walls_separate} 
    Let $X$ be a $\mathrm{CAT(0)}$ even 2-complex, and $\W$ a wall in $X$. Then 
    \begin{enumerate}
        \item The map $f:\W\to X$ that sends cells in $\W$ to their appropriate mirrors in $X$ is an isometric embedding. In particular, $\W$ can be identified with a closed and convex subset of $X$.
        \item Under this identification, $N_{\frac{1}{2}}(\W)$ convex, and isometric to $\W \times \left(-\frac{1}{2},\frac{1}{2}\right)$,
        \item $\W\subset X$ separates $X$ into two convex components, called (open) \textbf{halfspaces}.
    \end{enumerate}
\end{proposition}
\begin{proof}
        (1) Since $\W$ is connected and has finitely many isometry classes of cells, $\W$ is complete and geodesic.
    Hence, because $X$ is CAT(0), it suffices to show that $f$ is a local isometry \cite[Prop. II.4.14]{BridsonHaefliger1999}. 
    For any $x\in \W$, there is a neighborhood $U$ of $x$ such that $f(x)$ is contained in a neighborhood of the form $U\times\left(-\frac{1}{2},\frac{1}{2}\right)$, where $f$ maps $U$ into $U\times \{0\}$ via the identity map. Therefore, $f$ is a local isometry. 
    The fact that $X$ is uniquely geodesic gives us that $f(\W)$ is convex. Since $\W$ is complete, $f(\W)$ is closed in $X$. 

    (2) Since $\W$ is closed and convex, $N_{\frac{1}{2}}(\W)$ is convex. Furthermore, $f$ extends to a local isometry $\overline{f}:\W\times \left(-\frac{1}{2},\frac{1}{2}\right)\to N_{\frac{1}{2}}(\W)$ and since $N_{\frac{1}{2}}(\W)$ is CAT(0), $\overline{f}$ is an isometry. 

    (3) We apply the Mayer-Vietoris sequence to $U=X\setminus \W$ and $V=N_{\frac{1}{2}}(\W)$. Then $V$ is connected and $U\cap V$ is isometric to $\W\times \left(-\frac{1}{2},0\right)\sqcup \W\times \left(0,\frac{1}{2}\right)$. Since $U\cup V=X$ is simply connected, the short exact sequence 
    \[
        H_1(X)\to H_0(U\cap V)\to H_0(U)\oplus H_0(V)\to H_0(X) \to 0
    \]
    becomes
    \[
    0\to \Z^2\to H_0(U)\oplus \Z \to \Z \to 0
    \]
    and so $X\setminus \W$ has two connected components.
\end{proof}

In what follows, the wall $\W$ will always be identified with its image $f(\W)\subset X$. 

\begin{corollary}\label[corollary]{cor:isom_embedding}
    The combinatorial distance (denoted by $d_1$) between vertices $v,w\in X^{(0)}$ is equal to the number of walls in $S(v,w)$, where $S(v,w)$ is the set of walls separating $v$ from $w$.
\end{corollary}
\begin{proof}
    Since any edge path from $v$ to $w$ crosses the walls of $S(v,w)$ at least once, it suffices to check that $d_1(v,w)\leq \# S(v,w)$. 
    Let $\gamma$ be the CAT(0) geodesic from $v$ to $w$. Observe that $\gamma$ crosses every wall in $S(v, w)$ exactly once. Further, only finitely many points along $\gamma$ lie in a wall of $X$.
    Indeed, otherwise there would be a neighbourhood of $X$ where $\gamma$ lies entirely in a wall $\W$. This would imply that $\gamma$ lies entirely in $\W$ by \cref{prop:walls_separate}, contradicting the fact that $\gamma$ is a path between vertices in~$X$.

    In light of this observation, we can partition $\gamma$ into segments $\gamma_1,\dots,\gamma_n$ at the points where $\gamma$ intersects a wall. Now, since the $\gamma_i$ cross no wall, there is a unique vertex $v_i\in X^{(0)}$ contained in the same halfspaces as $\gamma_i$. 
    Then each $v_i$ and $v_{i+1}$ lie in the same common cell, and the walls separating them are exactly those which separate $\gamma_i$ from $\gamma_{i+1}$ in that cell. Thus, there is a path from $v_i$ to $v_{i+1}$ in the cell which crosses said walls exactly once. Concatenating the paths then gives a path from $v_1=v$ to $v_n=w$ of length $\# S(v,w)$.
\end{proof}

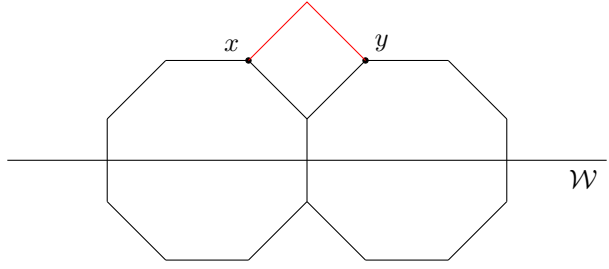
\begin{figure}
    \centering
    \centering
\begin{tikzpicture}
\def\len{1.1}
\def\nq{4}
\def\nl{((2*\nq)/(\nq-2))}

\def\rq{0.5*\len cm / sin{(180/(2*\nq))}}
\def\Rq{0.5*\len cm / tan{(180/(2*\nq))}}

\def\rl{0.5*\len cm / sin{(180/(2*\nl))}}
\def\Rl{0.5*\len cm / tan{(180/(2*\nl))}}

\foreach \k in {1,2,...,7} {
    \coordinate (v\k) at ($({180/(2*\nq)+(\k-1)*180/(\nq)}:\rq) - (\Rq, 0)$);
    \coordinate (q) at ($({180/(2*\nq)+\k*180/(\nq)}:\rq) - (\Rq, 0)$);
    
    \draw (v\k) -- (q);
}

\foreach \k in {0,1,...,7} {
    \coordinate (u\k) at ($({180/(2*\nl)+(\k-1)*180/(\nl)}:\rl) + (\Rl, 0)$);
    \coordinate (q') at ($({180/(2*\nl)+\k*180/(\nl)}:\rl) + (\Rl, 0)$);
    
    \draw (u\k) -- (q');
}

\filldraw (v2) circle (1pt) node [above left] {$x$};
\filldraw (u3) circle (1pt) node [above right] {$y$};

\draw [red] (v2) -- ($(v2)!1!90:(v1)$) -- (u3);

\draw (-3*\Rq, 0) -- (3*\Rl, 0) node [below left] {$\W$};

\end{tikzpicture}
    \unskip
    \caption{$N(\W)$ fails to be convex in both the CAT(0) metric and the combinatorial metric.}
    \label{fig:carriers_not_convex}
\end{figure}

The \textbf{carrier} of a wall $\W$, denoted by $N(\W)$ is defined to be the union of all cells in $X$ intersecting $\W$.  In contrast to cube complexes, the carrier is in general distinct from $N_{\frac{1}{2}}(\W)$, and may fail to be convex both with the CAT(0) metric and the combinatorial metric on $X^{(1)}$ (\cref{fig:carriers_not_convex}).

\begin{lemma}\label[lemma]{lem:carrier_isom_emb}
    The inclusion $N(\W)^{(1)}\hookrightarrow X^{(1)}$ is an isometric embedding.
\end{lemma}

\begin{proof}
    Given vertices $v,w\in N(\W)^{(0)}$, we may choose points $v',w'\in N_{\frac{1}{2}}(\W)$ that are contained in the same halfspaces as $v$ and $w$ respectively. By \cref{prop:walls_separate}(2), the CAT(0) geodesic $[v',w']$ is contained in $N_{\frac{1}{2}}(\W)$ and therefore the construction in the proof of \cref{cor:isom_embedding} applied to $[v',w']$ yields a path in $N(\W)^{(1)}$ from $v$ to $w$ of length $\#S(v,w)$.
\end{proof}

\begin{lemma}\label[lemma]{lem:uniqueness_comb_projections_onto_walls}
    Let $X$ be a $\mathrm{CAT(0)}$ even 2-complex and let $\W$ be a wall in $X$. Let $P\subset X$ be a 2-cell intersecting $\W$, and let $v$ be a vertex of $P$. Then any combinatorial geodesic from $v$ into~$\W$ of length $d_1(v, \W)$ is contained in $P$. 
\end{lemma}
\begin{proof}
    By \cref{lem:carrier_isom_emb}, it suffices to show that any combinatorial geodesic $\gamma$ in $N(\W)^{(1)}$ from $v$ into $\W$ that is not contained in $P$ has length $>d_1(v, \W)$. Suppose that $u, w\in \gamma$ are adjacent vertices with $u\in P$ and $w\notin P$. Then the edge $[u,w]$ is contained in some cell $P'\subset N(\W)$ intersecting $\W$ whose interior is disjoint from $P$. But since cells in $X$ are convex in the CAT(0) metric, $\pi_\W(u)$ is contained in $P\cap P'$. Therefore, $\pi_\W(u)$ is the midpoint of a common edge of $P$ and $P'$ dual to $\W$, and so $u$ is adjacent to $\W$. But then $d_1(u, \W)\leq d_1(w, \W)$ and thus $\gamma$ has length greater than~$d_1(v, \W)$.
\end{proof}

\section{Truncated piecewise Euclidean structures}
\label{sec:pw_euc_structs}

Given a regular  $2n$-gon of edge length 1, we may take its barycentric subdivision, and endow each triangle in the subdivision with the Euclidean metric coming from the triangles in the barycentric subdivision of a regular $2q$-gon of edge length 1. We call the resulting complex a \textbf{$2q$-truncated $2n$-gon} (see \cref{fig:truncation_example}). 
\begin{figure}
    \centering
    \begin{tikzpicture}
        \def\n{4}
        \def\len{1.5}
            
        \def\r{0.5*\len cm / sin{(180/(2*\n))}}
        \def\R{0.5*\len cm / tan{(180/(2*\n))}}

        \foreach \k in {-4,-3,...,3} {
            \coordinate (p) at ({180/(2*\n)+(\k-1)*180/(\n)}:\r);
            \coordinate (q) at ({180/(2*\n)+\k*180/(\n)}:\r);
            
            \draw[thick] (p) -- (q);
            \draw[gray] (q) -- (0,0) -- ($(p)!0.5!(q)$);
            }

        \coordinate (O) at (0,0);
        \coordinate (a) at (\R,0);
        \coordinate (b) at ({180/(2*\n)}:\r);

        \draw (O) -- (a) -- (b) -- (O);
            
        \path pic ["\small{$\frac{\pi}{6}$}", draw, angle radius=0.4*\len cm, angle eccentricity=1.5] {angle=a--O--b};
        \path pic ["\small{$\frac{\pi}{3}$}", draw, angle radius=0.18*\len cm, angle eccentricity=1.7] {angle=O--b--a};
        \path pic [draw, angle radius=0.1*\len cm, angle eccentricity=1.5] {right angle=O--a--b};
            
        \coordinate (v) at ({180/(2*\n)+(6-1)*180/(\n)}:\r);
        \draw (v) node [below left] {$v$};
        \coordinate (w) at ({180/(2*\n)+(3-1)*180/(\n)}:\r);
        \draw (w) node [above] {$w$};
            
        \draw[red] (v)--(O)--(w);
    \end{tikzpicture}
    \unskip
    \caption[6-truncated octagon]{A 6-truncated octagon. The path from $v$ to $w$ passing through the center is geodesic.}
    \label{fig:truncation_example}
\end{figure}
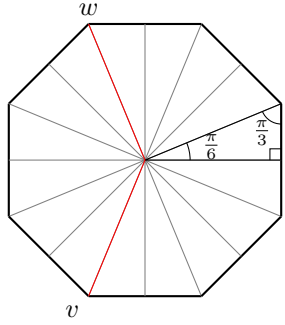
Consider the function $q:\N_{\geq 2}\to\N_{\geq 2}$ defined by 
\begin{equation}\label{eq:truncation_fct}
       q(n) = \begin{cases}
                n & \text{if } n=2,3\\
                4 & \text{if } n=4,5 \\
                6 & \text{if } n \geq 6\\
            \end{cases}
\end{equation}
Given a $\mathrm{CAT(0)}$ even 2-complex $X$, let $X'$ be the complex homeomorphic to $X$ constructed by replacing each $2n$-gon with a $2q(n)$-truncated $2n$-gon. Note that $X$ and $X'$ have isometric 1-skeletons. Taking the pullback of the piecewise Euclidean metric on $X'$ defines a new metric on $X$, called a \textbf{truncated piecewise Euclidean metric} on $X$.

\begin{lemma}[\cite{Munro2021}]
    \label[lemma]{lem:truncation}
    The piecewise Euclidean metric on $X'$ is $\mathrm{CAT(0)}$. Furthermore, walls in $X$ are convex with respect to the truncated piecewise Euclidean metric defined above.
\end{lemma}

\begin{proof}
    Since $X'$ is simply connected, it suffices to verify that the link of each vertex in $X'$ has no embedded cycles of length less than $2\pi$. 
    If $v\in X'$ is the midpoint of an edge in $X$, then $\op{Lk}(v,X')$ is isometric to $\op{Lk}(v,X)$.
    If $v\in X'$ is the center of a $2n$-gon, then $\op{Lk}(v,X')$ consists of a single cycle of length $\frac{\pi}{2q(n)}\cdot 4n$, which is at least $2\pi$ because $q(n)\leq n$.
    
    If $v\in X'$ is a vertex of $X$, then edges in $\op{Lk}(v, X')$ all have length at least $\pi/4$. As such, it suffices to look at $k$-cycles for $k<8$. But $k$-cycles in $\op{Lk}(v, X')$ correspond to $\frac{k}{2}$-cycles in $\op{Lk}(v, X)$, and since $X$ is CAT(0), these links are simplicial. Thus any cycle with less than 8 edges in $\op{Lk}(v, X')$ is a $6$-cycle.  Consider the associated $3$-cycle in $\op{Lk}(v,X)$ with edges of length $\frac{n_i-1}{n_i}\pi$ for $i=1,2,3$ (i.e.\ the interior angles of polygons incident to $v$). Then the cycle in $\op{Lk}(v, X')$ is of length
\[
    \frac{q(n_1)-1}{q(n_1)}\pi + \frac{q(n_2)-1}{q(n_2)}\pi +\frac{q(n_3)-1}{q(n_3)}\pi
\]
Therefore, to show that the link condition is satisfied, we need to check that 
\[
        \frac{1}{n_1} + \frac{1}{n_2} + \frac{1}{n_3}\leq 1\qquad\text{implies that}\qquad \frac{1}{q(n_1)} + \frac{1}{q(n_2)} + \frac{1}{q(n_3)}\leq 1
\]
for integers $n_1\geq n_2\geq n_3\geq 2$. If $n_3>2$, then the inequality is easily seen to be satisfied, so we may assume that $n_3=2$. The left hand inequality then gives us that $n_1,n_2\geq 3$. If $n_2=3$, then $n_1\geq 6$, in which case $\frac{1}{q(n_1)} + \frac{1}{q(n_2)} + \frac{1}{q(n_3)}=1$. Lastly, if $n_2\geq 4$, then 
\[
        \frac{1}{q(n_1)} + \frac{1}{q(n_2)} + \frac{1}{q(n_3)}\leq \frac{1}{4} + \frac{1}{4} + \frac{1}{2}=1
\]

Finally, geodesics connecting points in walls of $X$ with respect to the original metric on $X$ remain geodesic in the pullback metric on $X$.
Indeed, angles between segments starting at the center of a cell in $X$ are increased because $q(n)\leq n$, and walls remain perpendicular to their dual edges in~$X$.
\end{proof}

\begin{lemma}\label[lemma]{lem:angles_triangle_equality}
    Let $X$ be a 2-complex and $P\subset X$ be a 2-cell of $X$. Let $v$ be a vertex of $P$ with adjacent vertices $u_1,u_2\in P$. Fix $y\in P$ distinct from $v$ and let $x\in X$ be a point distinct from $v$ such that $[v, x]\cap P=\{v\}$. Then 
    \[
            \angle_v(y,x) = \min\Bigl\{ \angle_v(y,u_1)+\angle_v(u_1, x),\ \angle_v(y,u_2)+\angle_v(u_2, x),\ \pi \Bigr\}
    \]
\end{lemma}

\begin{proof}
    This follows from \cref{lem:angle_linkpath}, as any path in $\lk(v, X')$ between the point $\vec{x}$ corresponding to $[v,x]$ and the vertex $\vec{y}$ corresponding to $[v,y]$ passes through at least one of the vertices  $\vec{u}_i\in \lk(v, X')$ corresponding the edge $[v, u_i]$. 
\end{proof}

\begin{lemma}
    \label[lemma]{lem:bounding_2cells}
    Let $X$ be a $\mathrm{CAT(0)}$ even 2-complex, endowed with a truncated piecewise Euclidean metric.
    With the notation from \cref{lem:angles_triangle_equality}, suppose further that $x$ is a vertex adjacent to $v$, and $\W$ is a wall such that $[v, \pi_\W(v)]$ intersects $P$. 
    If the wall dual to the edge $[v, x]$ intersects $\W$, then one of the triples $u_i, v, x$ bound a $2p$-truncated polygon with $\frac{\pi}{p}>\angle_v(\pi_\W(v), u_i)$.
\end{lemma}

\begin{proof}
    Take any point $y\in [v,\pi_\W(v)]\cap P$. If the wall $\W'$ dual to $[v, x]$ intersects $\W$, then our disjointness criterion (\cref{lem:parallel_criterion}) implies that $\angle_v(\pi_\W(v), \pi_{\W'}(v)) = \angle_v(y, x)<\pi$. In particular, by \cref{lem:angles_triangle_equality}, at least one of the $u_i$ satisfies $\angle_v(y, u_i)+\angle_v(u_i, x)<\pi$. But paths between vertices in $\lk(v, X')$ corresponding to edges of $X$ have length at least $\pi/2$. It follows that $\vec{u}_i$ and $\vec{x}$ are adjacent in $\lk(v, X)$, and therefore $u_i,v,x$ bound a $2p$-truncated polygon. In particular,  we obtain $\angle_v(u_i, x)=\frac{p-1}{p}\pi<\pi - \angle_v(y, u_i)$, and so $\angle_v(\pi_\W(v), u_i)=\angle_v(y, u_i)<\frac{\pi}{p}$.
\end{proof}

\begin{example}[Large-type even 2-complexes]
    \label[example]{eg:large_type_par_wall_thm}
    Suppose $X$ is a CAT(0) even 2-complex containing no squares. Then the truncated piecewise Euclidean metric obtained by instead taking $q(n)=3$ for all $n$ still yields a CAT(0) metric.  Here we use this metric instead of the metric from \cref{lem:truncation}. Now, if $u_1 \neq u_2$ are vertices adjacent to a vertex $v$, then $\angle_v(u_1, u_2)\geq \frac{2\pi}{3}$.
    
    Let $\W$  be a wall and consider an edge with endpoints $v\in N(\W)$ and $w\notin N(\W)$.  With the truncated piecewise Euclidean metric the projection $\pi_\W(v)$ is either the midpoint of an edge or the center of a 6-truncated polygon $P$, depending on whether $v$ is adjacent to $\W$ or not. 

    In the latter case, let $u_1,u_2$ be the vertices of $P$ adjacent to $v$. Then $\angle_v(\pi_\W(v), u_i)=\frac{\pi}{3}$, and so the wall dual to the edge $[v,w]$ is disjoint from $\W$ by \cref{lem:bounding_2cells}.
    In the former case, $\angle_v(\pi_\W(v), w)<\pi$ if and only if $\pi_\W(v),v, w$ are contained in a common cell of $X$, which would contradict the assumption that $w\notin N(\W)$. Thus $\angle_v(\pi_\W(v), w)=\pi$, and therefore by \cref{lem:parallel_criterion} the wall dual to the edge $[v,w]$ is disjoint from $\W$.
    Therefore, any vertex outside the carrier of $\W$ is separated from $\W$ by a wall $\W'$. 
    
    Note that by \cref{lem:uniqueness_comb_projections_onto_walls}, any vertex in $N(\W)$ is at combinatorial distance $\leq \frac{N}{2}$, where $2N$ is the largest number of sides of a polygon in $\op{Shapes}(X)$. It thus follows that any vertex with $d_1(x, \W)>\frac{N}{2}$ is separated from $\W$ by another wall in $X$.
\end{example}

\section{The 2-dimensional Parallel Wall Theorem}
\label{sec:proof}

The arguments in the proof of \cref{thm:2D_par_wall} are the same in spirit to those of \cref{eg:large_type_par_wall_thm}. By \cref{lem:uniqueness_comb_projections_onto_walls}, the combinatorial distance between $\W$ and vertices in $N(\W)$ is again uniformly bounded by a constant depending only on $\op{Shapes}(X)$. Thus, we need only restrict our attention to vertices outside the carrier of $\W$. 
In this section, we aim to show that any vertex $x\notin N(\W)$ at combinatorial distance $\geq 5+\frac{1}{2}$ from $\W$ is separated from $\W$ by another wall in $X$.

We first pass to the CAT(0) truncated piecewise Euclidean metric on $X$ determined by \cref{eq:truncation_fct}, so that incident edges in 2-cells of $X$ have Alexandrov angle $\frac{q-1}{q}\pi$ for $q=2, 3, 4$ or $6$. This allows us to control the combinatorics of $X$ by making extensive use of \cref{lem:bounding_2cells}. We also repeatedly apply our disjointness criterion (\cref{lem:parallel_criterion}) in conjuction with the formula for Alexandrov angles obtained in \cref{lem:angles_triangle_equality}.

The first step in our proof is to reduce to the case where a path $\gamma$ from $x$ to $\W$ of length $d_1(x, \W)$ contains exactly 2 vertices of $N(\W)$, say $v_1$ and $v_2$  (\cref{lem:red_k=2_bounding_squares} and \cref{cor:first_cell}). Let now $v_3\notin N(\W)$ be the next vertex in~$\gamma$. 
The next step is to show that if the wall dual to the edge $[v_2,v_3]$ intersects $\W$, then the vertices $v_1, v_2, v_3$ bound either a hexagon or a square (\cref{lem:second_cell}).
Finally, we treat the hexagon and square cases separately in \cref{prop:hex_case} and \cref{prop:square_case}.

\begin{lemma}\label[lemma]{lem:red_k=2_bounding_squares}
    Let $X$ be a $\mathrm{CAT(0)}$ even 2-complex, $\W$ a wall in $X$, and let $P$ be a $2n$-gon contained in the carrier of $\W$. Suppose that $v\in P$ is a vertex at distance $\geq 2+\frac{1}{2}$ from $\W$ and that $w\notin N(\W)$ is a vertex adjacent to $v$. If the wall $\W'$ dual to $[v,w]$ intersects $\W$, then there is a vertex $u\in P$ adjacent to $v$ with $d_1(u,\W)<d_1(v, \W)$ such that $u, v, w$ bounds a square in $X$.
\end{lemma}

\begin{proof}
    \begin{figure}
    \centering
    \begin{subfigure}{0.45\textwidth}
    \centering
    \begin{tikzpicture}
        \def\n{6}
         \def\len{1.1}
            
        \def\r{0.5*\len cm / sin{(180/(2*\n))}}
        \def\R{0.5*\len cm / tan{(180/(2*\n))}}

        \foreach \k in {0,1,...,3} {
            \coordinate (p\k) at ({180/(2*\n)+(\k-1)*180/(\n)}:\r);
            \coordinate (q) at ({180/(2*\n)+\k*180/(\n)}:\r);
            
            \draw (p\k) -- (q);
        }

        \draw  ({-0.7*\R}, 0) -- ({1.5*\R}, 0) node[below left] {$\mathcal{W}$};

        \coordinate (m) at (\R, 0);
        \coordinate (C) at (0,0);

        \filldraw (m) circle (1.5pt) node [above right] {$m$};
        \filldraw (C) circle (1pt) node [below] {$c$};

        \coordinate (u1) at (q);
        \coordinate (v) at (p3);
        \coordinate (u2) at (p2);
        \coordinate (u3) at (p1);
        
        \filldraw (u1) circle (1pt) node [above left] {$u_1$};
        \filldraw (v) circle (1.5pt) node [above left] {$v$};
        \filldraw (u2) circle (1.5pt) node [above right] {$u_2$};
        \filldraw (u3) circle (1.5pt) ;

        \coordinate (w) at ($(v)!1!100:(u2)$);
        \filldraw (w) circle (1.5pt) node [right] {$w$};

        \draw [ultra thick] (w)--(v)--(u2)--(u2)--(p1)--(m); 

        \coordinate (m') at ($(v)!0.5!(w)$);
        \coordinate (W'a) at ($(m')!2!-90:(v)$);
        \coordinate (W'b) at ($(m')!4.5!90:(v)$);
        \draw (W'a)--(W'b) node [above=3pt] {$\W'$};
    
        \coordinate (pv) at ($(-1,0)!(v)!(1,0)$);
        \draw[dashed] (v) -- (pv);

        \path (C) -- (m) -- (u3)
            pic [draw, color=darkgray, angle radius=0.15*\len cm] {right angle=C--m--u3};
        \path (m) -- (u3) -- (u2)
            pic ["\small{$\frac{5\pi}{6}$}", draw, color=darkgray, angle radius=0.2*\len cm, angle eccentricity=2] {angle=u2--u3--m};
        \path (u3) -- (u2) -- (v)
            pic ["\small{$\frac{5\pi}{6}$}", draw, color=darkgray, angle radius=0.2*\len cm, angle eccentricity=2] {angle=v--u2--u3};
        \path (u2) -- (v) -- (pv)
            pic ["\small{$\frac{\pi}{3}$}", draw, color=darkgray, angle radius=0.225*\len cm, angle eccentricity=2] {angle=pv--v--u2};
        \path (v) -- (pv) -- (m)
            pic [draw, color=darkgray, angle radius=0.15*\len cm, angle eccentricity=2] {right angle=m--pv--v};
    \end{tikzpicture}
    \caption{Case where $\pi_\W(v)$ is not the center of $P$}
    \label{fig:red_to_k2_case1}
\end{subfigure}
\begin{subfigure}{0.45\textwidth}
    \centering
    \begin{tikzpicture}
        \def\n{6}
        \def\len{1.2}
            
        \def\r{0.5*\len cm / sin{(180/(2*\n))}}
        \def\R{0.5*\len cm / tan{(180/(2*\n))}}

        \coordinate (v) at ({180/(2*\n)+2*180/(\n)}:\r);
        \coordinate (u1) at ({180/(2*\n)+3*180/(\n)}:\r);
        \filldraw (v) circle (1.5pt) node [above right] {$v$};
        \filldraw (u1) circle (1pt) node [above left] {$u_1$};
        \filldraw ({180/(2*\n)}:\r) circle (1.5pt);
        \filldraw ({180-180/(2*\n)}:\r) circle (1pt);

        \coordinate (m) at (\R, 0);
        \coordinate (m') at (-\R, 0);
        \coordinate (C) at (0,0);
        
        \draw ({180/(2*\n) - 180/(\n)}:\r)
                --({180/(2*\n)}:\r);
        \draw ({180-180/(2*\n) + 180/(\n)}:\r)
                --({180-180/(2*\n)}:\r);
        \draw (v)--(u1);
        \draw [ultra thick] (m) -- ({180/(2*\n)}:\r) arc ({180/(2*\n)}:{180/(2*\n)+2*180/(\n)}:\r);
        \draw ({180-180/(2*\n)}:\r) arc ({180-180/(2*\n)}:{180/(2*\n)+3*180/(\n)}:\r);

        \draw  ({-1.5*\R}, 0) -- ({1.5*\R}, 0) node[below left] {$\mathcal{W}$};

        \filldraw (m) circle (1.5pt) node [above right] {$m$};
        \filldraw (m') circle (1pt) node [above left] {$m'$};
        
        \coordinate (w) at ($(v)!1!-90:(u1)$);
        \coordinate (w') at ($(u1)!1!90:(v)$);
        
        \draw (w)--(w')--(u1);
        \draw [ultra thick] (w)--(v);
        \filldraw (w) circle (1.5pt) node [right] {$w$};

        \coordinate (mv) at ($(v)!0.5!0:(w) + (\R,0)$);
        \coordinate (mu1) at ($(u1)!0.5!0:(w')-(\R,0)$);
        
        \draw (mu1) -- (mv) node [above left] {$\W'$};

        \coordinate (mu) at ($(v)!0.5!0:(u1)$);
        \filldraw (mu) circle (1pt) node [below right] {$\mu$};

        \draw[dashed] (C)--(mu)--(0, {\R+0.5*\len cm});

        \draw[gray] (u1)--(C)--(v);
        \path (v) -- (C) -- (m)
            pic ["\small{$\geq\! \frac{\pi}{2}$}", draw, color=darkgray, angle radius=0.25*\len cm, angle eccentricity=2] {angle=m--C--v};
        \path (u1) -- (C) -- (m')
            pic ["\small{$\geq\!\frac{\pi}{2}$}", draw, color=darkgray, angle radius=0.25*\len cm, angle eccentricity=2] {angle=u1--C--m'};

        \filldraw (C) circle (1pt) node [below] {$c$};
        \end{tikzpicture}
    \caption{Case where $\pi_\W(v)$ is the center of $P$}
    \label{fig:red_to_k2_case2}
\end{subfigure}
    \unskip
    \caption{Proof of \cref{lem:red_k=2_bounding_squares}}
    \label{fig:red_k=2_bounding_squares}
    \end{figure}
    Let $d_1(v, \W)=k+\frac{1}{2}$ for some $k\in \N_{\geq 2}$, and let $u_1,u_2$ be the vertices of $P$ adjacent to $v$. Suppose further that $d_1(u_1, \W)\geq d_1(u_2, \W)$. 
    By \cref{lem:uniqueness_comb_projections_onto_walls}, $d_1(v, \W)$ is achieved by a path $\gamma$ in $P$.
    In particular, $d_1(u_2, \W)<d_1(v, \W)$ and $k\leq \frac{n-1}{2}$. Since $k\geq 2$, we have $q(n)=4$ or $6$.
Let now $c$ be the center of~$P$, let $m$ be the midpoint of the edge in $\gamma$ intersecting $\W$ and let $m'$ be the midpoint of the other edge in $P$ intersecting $\W$. 
Then $\angle_c(v, m)=\min\Bigl\{(k+\frac{1}{2})\frac{\pi}{q(n)},\ \pi\Bigr\}$ and 
\[ 
\angle_c(v, m') = \min\Bigl\{\tfrac{n\pi}{q(n)}-(k+\tfrac{1}{2})\tfrac{\pi}{q(n)},\ \pi\Bigr\} \geq \tfrac{n\pi}{2q(n)}\geq \tfrac{\pi}{2}.
\]
Thus by \cref{prop:projection_existence_uniqueness}, $\pi_\W(v)$ is the center of $P$ if and only if $k\geq \frac{q(n)-1}{2}$.

If $\pi_\W(v)$ is not the center of $P$, then $k=2,q(n)=6$, and the vertices of $\gamma$, together with $m$ and $\pi_\W(v)$, bound a 5-sided Euclidean polygon (\cref{fig:red_to_k2_case1}). Since the sum of angles in a Euclidean pentagon is $3\pi$, we obtain
\[
    \angle_v(\pi_\W(v), u_2) = 3\pi-2\cdot\tfrac{5\pi}{6}-2\cdot\tfrac{\pi}{2}=\tfrac{\pi}{3}
\]
and $\angle_v(\pi_\W(v), u_1)=\tfrac{5\pi}{6}-\angle_v(\pi_\W(v), u_2)=\frac{\pi}{2}$. It then follows by \cref{lem:bounding_2cells} that $u_2, v, w$ bound a square, as desired.

If $\pi_\W(v)$ is the center of $P$, then $[v, \pi_\W(v)]$ is an edge of $X'$, and so $\angle_v(\pi_\W(v), u_i)=\frac{q(n)-1}{2q(n)}\pi> \frac{\pi}{3}$ (since $q(n)=4$ or 6). By  \cref{lem:bounding_2cells}, one of the $u_i, v, w$ bound a square. It remains to show that if $u_1, v, w$ bound a square~$C$, then $d_1(u_1, \W)<d_1(v, \W)$. Indeed, otherwise $\pi_\W(u_1)=c$ and thus the projection of every point in the edge $[v, u_1]$ onto $\W$ is also $c$ (\cref{fig:red_to_k2_case2}). 
But then taking $\mu$ to be the midpoint of the edge $[u_1, v]$, we see that $\pi_{\W'}(\mu)$ is the center of the square $C$ and  thus $\angle_\mu(\pi_\W(\mu), \pi_{\W'}(\mu))=\pi$, contradicting the assumption that $\W'$ intersects $\W$.
\end{proof}

\begin{corollary}\label[corollary]{cor:first_cell}
    Suppose $x\notin N(\W)$ is a vertex that is not separated from $\W$ by another wall of $X$.
    Then there is a combinatorial geodesic $\gamma$ from $x$ to $\W$ of length $d_1(x, \W)$ containing exactly two vertices of $N(\W)$.
\end{corollary}

\begin{proof}
  Let $\gamma$ be a path from $x$ to $\gamma$ of length $d_1(x, \W)$ that has a minimal number of vertices in $N(\W)$. By \cref{lem:uniqueness_comb_projections_onto_walls}, these vertices are contained in a common 2-cell $P$. If $\gamma$ contained more than 2 vertices of $P$, then \cref{lem:red_k=2_bounding_squares} would contradict the minimality of $\gamma$. If instead $\gamma$ contained exactly one vertex of $P$, then, by the same argument as in \cref{eg:large_type_par_wall_thm}, the vertex $x$ would be separated from $\W$ by a wall dual to an edge in $\gamma$.
\end{proof}

Let us now fix some notation:
Let $\gamma$ be the combinatorial geodesic from $x$ to $\W$ obtained by \cref{cor:first_cell} and enumerate its vertices as $v_1,\dots,v_N=x$ in ascending order with respect to $d_1(v_j, \W)$.
In particular, $v_1$ and $v_2$ bound a $2q$-truncated polygon $P_q\subset N(\W)$ with $q>2$ and $v_3\notin N(\W)$.
We denote by $\W_j$ the wall dual to the edge $[v_j,v_{j+1}]$.
For convenience, we also define $\alpha \lplus \beta :=\min\{\alpha+\beta, \pi\}$.

\begin{figure}
    \centering
    \begin{tikzpicture}
        \def\n{4.5}
        \def\len{1.5}
            
        \def\r{0.5*\len cm / sin{(180/(2*\n))}}
        \def\R{0.5*\len cm / tan{(180/(2*\n))}}

        \foreach \k in {0,1,...,2} {
            \coordinate (p\k) at ({180/(2*\n)+(\k-1)*180/(\n)}:\r);
            \coordinate (q) at ({180/(2*\n)+\k*180/(\n)}:\r);
            
            \draw (p\k) -- (q);
            \filldraw (q) circle (1pt);
        }

        \filldraw (p1) circle (1.5pt) node [below right] {$v_1$};
        \filldraw (p2) circle (1.5pt) node [above left] {$v_2$};
        \filldraw (q) circle (1pt) node [left] {$u_2$};

        \draw  ({-0.5*\R}, 0) -- ({1.9*\R}, 0) node[below] {$\mathcal{W}$};

        \coordinate (u') at ($(p1)!1!100:(p0)$);
        \draw (p1) -- (u');
        \filldraw (u') circle (1pt) node [below] {$u_1$};
        
        \coordinate (v3) at ($(p2)!1!130:(p1)$);
        \draw [ultra thick] (\R,0) -- (p1) -- (p2) -- (v3);
        \filldraw (v3) circle (1.5pt) node [left] {$v_3$};

        \coordinate (m2) at ($(p2)!0.5!130:(p1)$);
        \coordinate (W2a) at ($(m2)!2!90:(v3)$);
        \coordinate (W2b) at ($(m2)!2.5!-90:(v3)$);
        \draw (W2a) -- (W2b) node [right] {$\W_2$};
        
        \coordinate (pv2) at ($(-1,0)!(p2)!(1,0)$);
        \draw [dashed] (p2)--(pv2);

        \coordinate (pv1) at ($(W2a)!(p1)!(W2b)$);
        \draw [dashed] (p1)--(pv1)  node [above, gray] {$P_p$};

        \path (pv2)--(p2)--(p1)
            pic ["\small{$\frac{\pi}{q}$}", draw, color=darkgray, angle radius=0.25*\len cm, angle eccentricity=1.7] {angle=pv2--p2--p1};

            \path (pv2)--(p2)--(q)
            pic ["\small{$\frac{q-2}{q}\pi$}", draw, color=darkgray, angle radius=0.3*\len cm, angle eccentricity=1.7] {angle=q--p2--pv2};

        \path (p2)--(p1)--(p0)
            pic ["\small{$\frac{q-1}{q}\pi$}", draw, color=darkgray, angle radius=0.15*\len cm, angle eccentricity=2.7] {angle=p2--p1--p0};

        \path (p2)--(p1)--(pv1)
            pic ["\small{$\frac{\pi}{p}$}", draw, color=darkgray, angle radius=0.25*\len cm, angle eccentricity=1.7] {angle=pv1--p1--p2};

        \path (pv1)--(p1)--(u')
            pic ["\small{$\frac{p-1}{p}\pi$}", draw, color=darkgray, angle radius=0.3*\len cm, angle eccentricity=1.7] {angle=u'--p1--pv1};

        \coordinate (C) at (0,0) node [below, gray] {$P_q$};
    \end{tikzpicture}
    \unskip
    \caption{The proof of \cref{lem:second_cell}}
    \label{fig:second_cell}
\end{figure}

\begin{lemma}\label[lemma]{lem:second_cell}
    We may  further choose $\gamma$ so that $v_1, v_2, v_3$ bound either a square or a hexagon $P_p$.
\end{lemma}

\begin{proof}
    Let $u_2\neq v_1$ be the second vertex in $P_q$ adjacent to~$v_2$. Then $\angle_{v_2}(\pi_\W(v_2), u_2)=\frac{q-2}{q}\pi\geq \frac{\pi}{2}$ unless $q=3$, or equivalently, unless $P_q$ is hexagon (\cref{fig:second_cell}). Therefore, up to interchanging $v_1$ and $u_2$ in $\gamma$, we may assume by \cref{lem:bounding_2cells} that $v_1, v_2, v_3$ bound a $2p$-truncated polygon $P_p$ for some~$p<q\leq 6$.
Let $u_1\neq v_2$ be the other vertex of $P_p$ adjacent to $v_1$. Then by  \cref{lem:angles_triangle_equality}, 
\[
    \angle_{v_1}(\pi_\W(v_1), \pi_{\W_2}(v_1))=\min\Bigl\{\tfrac{q-1}{q}\pi\lplus \tfrac{\pi}{p},\ \angle_{v_1}(\pi_\W(v_1), u_1)\lplus \tfrac{p-2}{p}\pi\Bigr\} = \angle_{v_1}(\pi_\W(v_1), u_1)\lplus \tfrac{p-2}{p}\pi
\]
Since $\angle_{v_1}(\pi_\W(v_1), u_1)\geq \frac{\pi}{2}$ and $\W_2$ intersects $\W$, we conclude by that $p<4$, i.e.\ that $P_p$ is either a hexagon or a square.
\end{proof}

\begin{proposition}[Hexagon case]\label[proposition]{prop:hex_case}
    Suppose $P_p$ is a hexagon. Then $d_1(x, \W)\leq 5+\frac{1}{2}$.
\end{proposition}

\begin{proof}
Recall that $q>2$, by minimality of the length of $\gamma$. Thus by \cref{lem:angles_triangle_equality}, we have $\angle_{v_1}(\pi_\W(v_1), \pi_{\W_2}(v_1))=\angle_{v_1}(\pi_\W(v_1), u_1)\lplus \frac{\pi}{3}$ (\cref{fig:case_m3_W2_bounding_square}).
Because $\W_2$ intersects $\W$,  we therefore have $\angle_{v_1}(\pi_\W(v_1), u_1)<\frac{2\pi}{3}$, and so 
$\pi_\W(v_1), v_1,u_1$ are contained in a square. This implies by the link condition that $q=6$. Furthermore, by minimality of the length of $\gamma$, the vertex $v_4$ does not belong to $P_p$.

\begin{figure}
    \centering
    \begin{subfigure}{0.3\textwidth}
    \centering
    \begin{tikzpicture}
    \def\n{6}
    \def\len{1.05}
        
    \def\r{0.5*\len cm / sin{(180/(2*\n))}}
    \def\R{0.5*\len cm / tan{(180/(2*\n))}}

    \foreach \k in {-1,0,...,2} {
        \coordinate (p\k) at ({180/(2*\n)+(\k-1)*180/(\n)}:\r);
        \coordinate (q) at ({180/(2*\n)+\k*180/(\n)}:\r);
        
        \draw (p\k) -- (q);
    }

    \filldraw (p1) circle (1.5pt) node [below right] {$v_1$};
    \filldraw (p2) circle (1.5pt) node [below left] {$v_2$};
    \filldraw (q) circle (1pt) node [left] {$u_2$};

    \draw  (0, 0) -- ({1.9*\R}, 0) node[below] {$\mathcal{W}$};

    \coordinate (u') at ($(p1)!1!90:(p0)$);
    \draw (p1) -- (u');
    \filldraw (u') circle (1pt) node [below] {$u_1$};
    
    \coordinate (v3) at ($(p2)!1!120:(p1)$);
    \draw [ultra thick] (\R,0) -- (p1) -- (p2) -- (v3);
    \filldraw (v3) circle (1.5pt) node [left] {$v_3$};

    \coordinate (tmp1) at (p2);
    \coordinate (tmp2) at (v3);
    \foreach \k in {0,1,...,2} {
        \coordinate (tmp3) at ($(tmp2)!1!120:(tmp1)$);
        \draw (tmp2) -- (tmp3);

        \coordinate (tmp1) at (tmp2);
        \coordinate (tmp2) at (tmp3);
    }
            
    \coordinate (m2) at ($(p2)!0.5!120:(p1)$);
    \coordinate (W2a) at ($(m2)!1.5!90:(v3)$);
    \coordinate (W2b) at ($(m2)!4.5!-90:(v3)$);
    \draw (W2a) -- (W2b) node [above] {$\W_2$};
    
    \coordinate (pv1) at ($(W2a)!(p1)!(W2b)$);
    \draw [dashed] (p1)--(pv1)  node [above, gray] {$P_p$};

    \path (p2)--(p1)--(pv1)
        pic ["\small{$\frac{\pi}{3}$}", draw, color=darkgray, angle radius=0.25*\len cm, angle eccentricity=1.7] {angle=pv1--p1--p2};

    \path (pv1)--(p1)--(u')
        pic ["\small{$\frac{\pi}{3}$}", draw, color=darkgray, angle radius=0.25*\len cm, angle eccentricity=1.7] {angle=u'--p1--pv1};

    \coordinate (C) at (0,0) node [below right, gray] {$P_q$};
\end{tikzpicture}
\caption{$\pi_\W(v_1), v_1, u_1$~bound a square}
\label{fig:case_m3_W2_bounding_square}

\end{subfigure}
\begin{subfigure}{0.3\textwidth}
\centering
\begin{tikzpicture}
    \def\n{6}
    \def\len{1.05}
        
    \def\r{0.5*\len cm / sin{(180/(2*\n))}}
    \def\R{0.5*\len cm / tan{(180/(2*\n))}}

    \foreach \k in {-1,0,...,2} {
        \coordinate (p\k) at ({180/(2*\n)+(\k-1)*180/(\n)}:\r);
        \coordinate (q) at ({180/(2*\n)+\k*180/(\n)}:\r);
        
        \draw (p\k) -- (q);
    }

    \filldraw (p1) circle (1.5pt) node [below right] {$v_1$};
    \filldraw (p2) circle (1.5pt) node [right] {$v_2$};
    \coordinate (p3) at (q);
    \filldraw (p3) circle (1pt) node [below] {$u_2$};

    \draw  (0, 0) -- ({1.9*\R}, 0) node[below] {$\mathcal{W}$};

    \coordinate (u') at ($(p1)!1!90:(p0)$);
    \draw (p1) -- (u');
    
    \coordinate (v3) at ($(p2)!1!120:(p1)$);
    \draw [ultra thick] (\R,0) -- (p1) -- (p2) -- (v3);
    \filldraw (v3) circle (1.5pt) node [above right] {$v_3$};

    \coordinate (tmp1) at (p2);
    \coordinate (tmp2) at (v3);
    \foreach \k in {0,1,...,2} {
        \coordinate (tmp3) at ($(tmp2)!1!120:(tmp1)$);
        \draw (tmp2) -- (tmp3);

        \coordinate (tmp1) at (tmp2);
        \coordinate (tmp2) at (tmp3);
    }
    \node [gray] (hex_center) at ($(tmp2)!1!60:(tmp1)$) {$P_p$};

    \draw (\R, -0.5*\len cm) -- ({\R+1*\len cm}, -0.5*\len cm) -- (u');

    \coordinate (v4) at ($(v3)!1!-130:(p2)$);
    \draw [ultra thick] (v3)--(v4);
    \filldraw (v4) circle (1.5pt) node [left] {$v_4$};
    
    \coordinate (m3) at ($(v3)!0.5!-130:(p2)$);
    \coordinate (W3a) at ($(m3)!1.8!95:(v4)$);
    \coordinate (W3b) at ($(m3)!2!-85:(v4)$);
    \draw (W3a) -- (W3b) node [below right] {$\W_3$};

    \coordinate (u2') at ($(p2)!1!-15:(p3)$);
    \draw (p2) -- (u2');
    \filldraw (u2') circle (1pt) node [above left] {$u_2'$};
    
    \coordinate (pv2) at ($(-1,0)!(p2)!(1,0)$);
    \draw [dashed] (p2)--(pv2);

    \coordinate (pv3) at ($(-1,0)!(v3)!(1,0)$);
    \draw [dashed] (v3)--(pv3);

    \coordinate (pv2_W3) at ($(W3a)!(p2)!(W3b)$);
    \draw [dashed] (p2)--(pv2_W3)  node [above, gray] {$P_\ell$};

    \path (pv2)--(p2)--(p1)
        pic ["\small{$\frac{\pi}{6}$}", draw, color=darkgray, angle radius=0.35*\len cm, angle eccentricity=1.7] {angle=pv2--p2--p1};

    \path (p3)--(p2)--(pv2)
        pic ["\small{$\frac{2\pi}{3}$}", draw, color=darkgray, angle radius=0.25*\len cm, angle eccentricity=1.7] {angle=p3--p2--pv2};

    \path (v3)--(p2)--(pv2_W3)
        pic ["{$\frac{\pi}{\ell}$}", draw, color=darkgray, angle radius=0.3*\len cm, angle eccentricity=1.7] {angle=v3--p2--pv2_W3};

    \path (p2)--(v3)--(pv3)
        pic ["{$\frac{\pi}{6}$}", draw, color=darkgray, angle radius=0.3*\len cm, angle eccentricity=1.75] {angle=p2--v3--pv3};

    \path (pv3) -- (v3) -- ($(v3)!1!120:(p2)$) coordinate (tmp)
        pic [draw, angle radius=0.15*\len cm, angle eccentricity=1.7] {right angle=pv3--v3--tmp};

    \coordinate (C) at (0,0) node [below right, gray] {$P_q$};
\end{tikzpicture}
\caption{$P_\ell$ must be a square}
\label{fig:case_m=3_a}
\end{subfigure}
\begin{subfigure}{0.3\textwidth}
\centering
\begin{tikzpicture}
    \def\len{1.05}
    \def\n{6}
    
    \def\r{0.5*\len cm / sin{(180/(2*\n))}}
    \def\R{0.5*\len cm / tan{(180/(2*\n))}}

    \coordinate (v1) at ({180/(2*\n)}:\r);
    \coordinate (v2) at ({3*180/(2*\n)}:\r);
    \coordinate (v3) at ($(v2)!1!120:(v1)$);
    \filldraw (v1) circle (1pt);
    \filldraw (v2) circle (1pt);
    \filldraw (v3) circle (1pt) node [above right] {$v_3$};

    \foreach \k in {-1,0,...,3}
        \draw ({180/(2*\n)+(\k-1)*180/(\n)}:\r) -- ({180/(2*\n)+\k*180/(\n)}:\r);
        
    \draw (-\r-0.3*\r+\R, 0) -- ({1.9*\R}, 0)  node[below] {$\mathcal{W}$};

    \coordinate (h4) at ($(v3)!1!120:(v2)$);
    \coordinate (h5) at ($(h4)!1!120:(v3)$);
    \coordinate (h6) at ($(h5)!1!120:(h4)$);
    
    \draw (v2) --(v3) -- (h4) -- (h5) -- (h6) -- (v1);
    
    \coordinate (v4) at ($(v3)!1!-90:(v2)$);
    \draw (v3) -- (v4);
    \filldraw (v4) circle (1pt) node [left] {$v_4$};

    \coordinate (p4) at ($(-1, 0)!(v4)!(0, 0)$);
    \draw[dashed] (v4) -- (p4);
    \path (p4) -- (v4) -- (v3)
        pic ["$\frac{\pi}{3}$", draw, color=darkgray, angle radius=0.33*\len cm, angle eccentricity=2] {angle=p4--v4--v3};
    
    \coordinate (c3) at ($(v2)!1!-5*180/6:(v1)$);
    \filldraw (c3) circle (1pt) node [above left] {$u_2$};
    \draw (v4) -- (c3);

    \draw (\R, -0.5*\len cm) -- ({\R+1*\len cm}, -0.5*\len cm) -- (h6);
    
    \draw[very thick] (\R,0) -- (v1) -- (v2) -- (v3) -- (v4);

\coordinate (C) at (0,0) node [below right, gray] {$P_q$};
\end{tikzpicture}
\caption{}
\label{fig:case_m=3_b}
\end{subfigure}
    \unskip
    \caption{Hexagon case}
\end{figure}
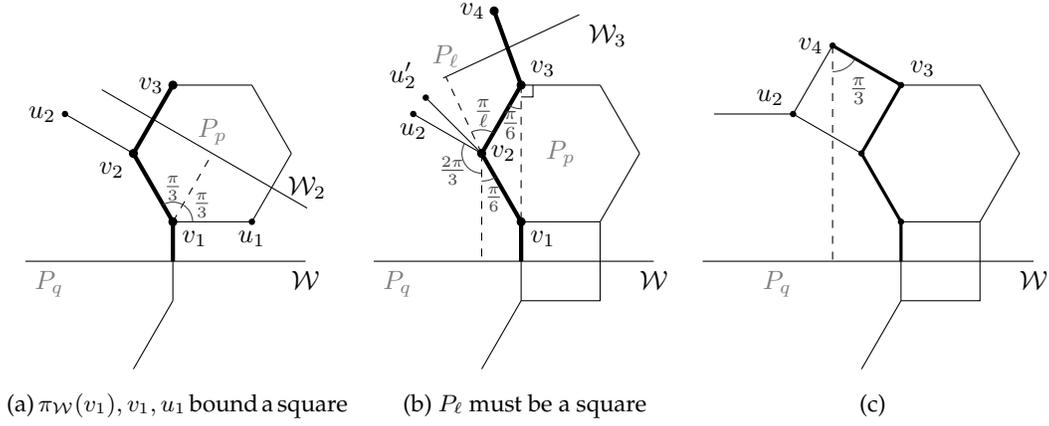

Now $\pi_{\W}(v_3)=\pi_\W(v_1)$, and $\angle_{v_2}(v_1, \pi_\W(v_2)) = \angle_{v_3}(v_2, \pi_\W(v_3))=\frac{\pi}{6}$ (\cref{fig:case_m=3_a}). In particular, applying \cref{lem:bounding_2cells} to the cell $P_p$ shows that $\W_3$ can only intersect $\W$ if $v_2, v_3$ and $v_4$ bound a $2\ell$-truncated polygon $P_\ell$, where $\ell<6$. Let $u_2'\neq v_3$ be the other vertex in $P_\ell$ adjacent to $v_2$. By the link condition, $\angle_{v_2}(v_1, u_2')\geq \frac{\pi}{3}\lplus\frac{\pi}{\ell}$ and $\angle_{v_2}(u_2, v_3)\geq \frac{\pi}{6}+\frac{\pi}{3}$. Thus, applying \cref{lem:angles_triangle_equality} first to the cell $P_q$, then to the cell $P_\ell$ gives
\[
    \angle_{v_2}(\pi_\W(v_2), \pi_{\W_3}(v_2))= \min
    \left\{\begin{array}{l}
        \frac{\pi}{6}\lplus \frac{2\pi}{3}\lplus \frac{\pi}{\ell},\\
        \frac{\pi}{6}\lplus\angle_{v_2}(v_1, u_2')\lplus \frac{\ell-2}{\ell}\pi,\\
        \frac{2\pi}{3}\lplus \angle_{v_2}(u_2, v_3)\lplus \frac{\pi}{\ell},\\
        \frac{2\pi}{3}\lplus \angle_{v_2}(u_2, u_2')\lplus \frac{\ell-2}{\ell}\pi
    \end{array}\right\} = \tfrac{2\pi}{3}\lplus \angle_{v_2}(u_2, u_2')\lplus \tfrac{\ell-2}{\ell}\pi
\]
Therefore, if $\W_3$ intersects $\W$, then by \cref{lem:parallel_criterion}, we have $u_2=u_2'$ and $\ell=2$.
In other words, the vertices $u,v_2,v_3, v_4$ bound a square. Note that $\angle_{v_4}(v_3, \pi_\W(v_4))=\frac{\pi}{3}$ (\cref{fig:case_m=3_b}).
 
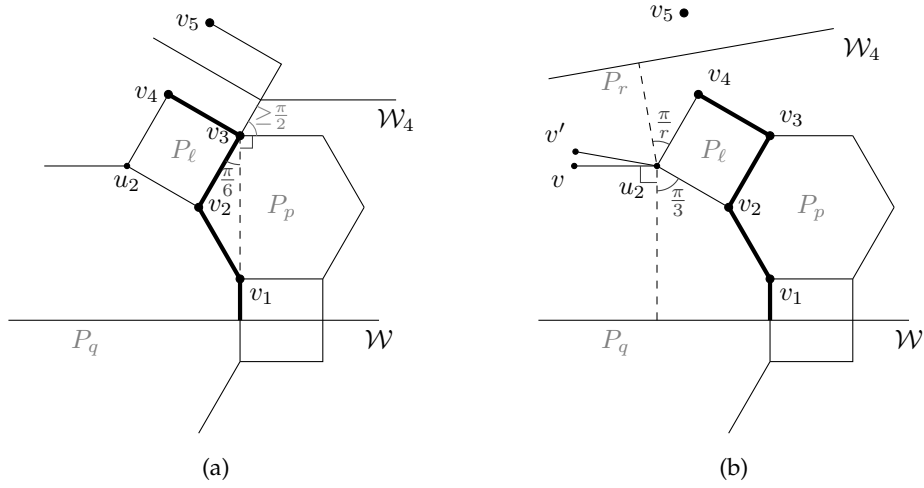
\begin{figure}
    \centering 
    \def\len{1.1}
\begin{subfigure}{0.45\textwidth}
    \centering
    \begin{tikzpicture}
    \def\n{6}
        
    \def\r{0.5*\len cm / sin{(180/(2*\n))}}
    \def\R{0.5*\len cm / tan{(180/(2*\n))}}

    \foreach \k in {-1,0,...,3} {
        \coordinate (p\k) at ({180/(2*\n)+(\k-1)*180/(\n)}:\r);
        \coordinate (q) at ({180/(2*\n)+\k*180/(\n)}:\r);
        
        \draw (p\k) -- (q);
    }

    \filldraw (p1) circle (1.5pt) node [below right] {$v_1$};
    \filldraw (p2) circle (1.5pt) node [right] {$v_2$};
    \filldraw (p3) circle (1pt) node [below] {$u_2$};

    \draw  ({-0.5*\R}, 0) -- ({1.9*\R}, 0) node[below] {$\mathcal{W}$};

    \coordinate (u') at ($(p1)!1!90:(p0)$);
    \draw (p1) -- (u');
    
    \coordinate (v3) at ($(p2)!1!120:(p1)$);
    \draw [ultra thick] (\R,0) -- (p1) -- (p2) -- (v3);
    \filldraw (v3) circle (1.5pt) node [ left] {$v_3$};

    \coordinate (tmp1) at (p2);
    \coordinate (tmp2) at (v3);
    \foreach \k in {0,1,...,2} {
        \coordinate (tmp3) at ($(tmp2)!1!120:(tmp1)$);
        \draw (tmp2) -- (tmp3);

        \coordinate (tmp1) at (tmp2);
        \coordinate (tmp2) at (tmp3);
    }
    \node [gray] (hex_center) at ($(tmp2)!1!60:(tmp1)$) {$P_p$};

    \draw (\R, -0.5*\len cm) -- ({\R+1*\len cm}, -0.5*\len cm) -- (u');

    \coordinate (v4) at ($(v3)!1!-90:(p2)$);
    \draw [ultra thick] (v3)--(v4);
    \filldraw (v4) circle (1.5pt) node [left] {$v_4$};

    \draw (v4)--(p3);
    \node [gray] (hex_center) at ($(v3)!0.7!45:(v4)$) {$P_\ell$};

    \coordinate (v5) at ($(v4)!1!90:(v3)$);
    \filldraw (v5) circle (1.5pt) node [left] {$v_5$};

    \draw (v5) -- ($(v5)!1!90:(v4)$) -- (v3);
    
    \coordinate (m4) at ($(v4)!0.5!0:(v5)$);
    \coordinate (tmp) at ($(m4)!1!-90:(v4)$);
    \coordinate (m4') at ($(v3)!0.5!-90:(v4)$);
    \draw (tmp) -- (m4) -- (m4') -- ($(m4')+(1.8,0)$) node [below] {$\W_4$};

    \coordinate (pv3) at ($(-1,0)!(v3)!(1,0)$);
    \draw [dashed] (v3)--(pv3);

    \path (p2)--(v3)--(pv3)
        pic ["{$\frac{\pi}{6}$}", draw, color=darkgray, angle radius=0.35*\len cm, angle eccentricity=1.65] {angle=p2--v3--pv3};

    \path (pv3) -- (v3) -- ($(v3)!1!120:(p2)$) coordinate (tmp)
        pic [draw, color=darkgray, angle radius=0.15*\len cm, angle eccentricity=1.7] {right angle=pv3--v3--tmp};

    \path (m4') -- (v3) -- ($(v3)+(0.5,0)$) coordinate (tmp)
        pic ["\small{$\geq\!\!\frac{\pi}{2}$}",draw, color=gray, angle radius=0.2*\len cm, angle eccentricity=2.25] {angle=tmp--v3--m4'};
    
    \coordinate (C) at (0,0) node [below, gray] {$P_q$};
\end{tikzpicture}
    \caption{}
    \label{fig:case_m3_W4_square}
\end{subfigure}
\begin{subfigure}{0.45\textwidth}
\centering
\begin{tikzpicture}
    \def\n{6}
    \def\r{0.5*\len cm / sin{(180/(2*\n))}}
    \def\R{0.5*\len cm / tan{(180/(2*\n))}}

    \foreach \k in {-1,0,...,3} {
        \coordinate (p\k) at ({180/(2*\n)+(\k-1)*180/(\n)}:\r);
        \coordinate (q) at ({180/(2*\n)+\k*180/(\n)}:\r);
        
        \draw (p\k) -- (q);
    }

    \filldraw (p1) circle (1.5pt) node [below right] {$v_1$};
    \filldraw (p2) circle (1.5pt) node [right] {$v_2$};
    \filldraw (p3) circle (1pt);

    \coordinate (v) at (q);
    \filldraw (v) circle (1pt) node [below left] {$v$};
    
    \draw  ({-0.5*\R}, 0) -- ({1.9*\R}, 0) node[below] {$\mathcal{W}$};

    \coordinate (u') at ($(p1)!1!90:(p0)$);
    \draw (p1) -- (u');
    
    \coordinate (v3) at ($(p2)!1!120:(p1)$);
    \draw [ultra thick] (\R,0) -- (p1) -- (p2) -- (v3);
    \filldraw (v3) circle (1.5pt) node [above right] {$v_3$};

    \coordinate (tmp1) at (p2);
    \coordinate (tmp2) at (v3);
    \foreach \k in {0,1,...,2} {
        \coordinate (tmp3) at ($(tmp2)!1!120:(tmp1)$);
        \draw (tmp2) -- (tmp3);

        \coordinate (tmp1) at (tmp2);
        \coordinate (tmp2) at (tmp3);
    }
    \node [gray] (hex_center) at ($(tmp2)!1!60:(tmp1)$) {$P_p$};

    \draw (\R, -0.5*\len cm) -- ({\R+1*\len cm}, -0.5*\len cm) -- (u');

    \coordinate (v4) at ($(v3)!1!-90:(p2)$);
    \draw [ultra thick] (v3)--(v4);
    \filldraw (v4) circle (1.5pt) node [above right] {$v_4$};

    \draw (v4)--(p3);
    \node [gray] (hex_center) at ($(v3)!0.7!45:(v4)$) {$P_\ell$};

    \coordinate (v5) at ($(v4)!1!130:(v3)$);
    \filldraw (v5) circle (1.5pt) node [left] {$v_5$};
    
    \coordinate (m4) at ($(v4)!0.5!0:(v5)$);
    \coordinate (W4a) at ($(m4)!3.5!-90:(v4)$);
    \coordinate (W4b) at ($(m4)!3.5!90:(v4)$);
    \draw (W4a) -- (W4b) node [below right] {$\W_4$};

    \coordinate (v') at ($(p3)!1!-10:(v)$);
    \draw (p3) -- (v');
    \filldraw (v') circle (1pt) node[above left] {$v'$};
    
    \coordinate (pu2) at ($(-1,0)!(p3)!(1,0)$);
    \coordinate (pu2_W4) at ($(W4a)!(p3)!(W4b)$);
    
    \draw[dashed] (pu2)--(p3)--(pu2_W4) node [gray, below left] {$P_r$};

    \path (pu2)--(p3)--(p2)
        pic ["{$\frac{\pi}{3}$}", draw, color=darkgray, angle radius=0.3*\len cm, angle eccentricity=1.7] {angle=pu2--p3--p2};

    \path (pu2_W4)--(p3)--(v4)
        pic ["{$\frac{\pi}{r}$}", draw, color=darkgray, angle radius=0.3*\len cm, angle eccentricity=1.7] {angle=v4--p3--pu2_W4};

    \path (v) -- (p3) -- (pu2)
        pic ["{\color{black}$u_2$}", draw, color=darkgray, angle radius=0.2*\len cm, angle eccentricity=1.5] {right angle=v--p3--pu2};
    \coordinate (C) at (0,0) node [below, gray] {$P_q$};
\end{tikzpicture}
    \caption{}
    \label{fig:case_m3_W4_square_octagon}
\end{subfigure}
    \unskip
    \caption{Hexagon case (continued)}
    \label{fig:case_m3_W4}
\end{figure}
 
We now assume that $\W_4$ intersects $\W$. This means that either $v_3, v_4,v_5$ bound a square or $u, v_4, v_5$ bound a $2r$-truncated polygon $P_r$ with $r<6$. In the former case, we have that $\angle_{v_3}(\pi_\W(v_3), \pi_{\W_4}(v_3))=\pi$ (\cref{fig:case_m3_W4_square}). In the latter case, let $v\neq v_2$ be the second vertex in $P_q$ adjacent to $u_2$, and let $v'\neq v_4$ be the second vertex in $P_r$ adjacent to $u_2$ (\cref{fig:case_m3_W4_square_octagon}). Then $\angle_{u_2}(v, v_4)\geq \frac{\pi}{2}$ because $v\neq v_4$, and since $r<6$, we have by the link condition that $\angle_{u_2}(v_2, v')\geq \frac{3\pi}{4}$. Thus applying \cref{lem:angles_triangle_equality}
once to the cell $P_q$ and again to the cell $P_r$, we obtain
\[
    \angle_{u_2}(\pi_\W(u), \pi_{\W_4}(u))=\min
    \left\{\begin{array}{l}
        \tfrac{\pi}{2} \lplus\angle_{u_2}(v,v')\lplus \frac{r-2}{r}\pi,  \\
        \tfrac{\pi}{2} \lplus\angle_{u_2}(v, v_4)\lplus \frac{\pi}{r}, \\
        \tfrac{\pi}{3}\lplus \tfrac{\pi}{2}\lplus \tfrac{\pi}{r},\\
        \tfrac{\pi}{3} \lplus\angle_{u_2}(v_2,v')\lplus \frac{r-2}{r}\pi  \\
    \end{array}\right\} = \tfrac{\pi}{2} \lplus\angle_{u_2}(v,v')\lplus \tfrac{r-2}{r}\pi
\]
In particular, if $\W_4$ intersects $\W$, then $v=v'$ and $r<4$. But the link condition implies that $P_r$ cannot be a square sharing an edge with $P_q$, so $r=3$. We claim then that $\W_5$ is disjoint from $\W$. Here, there are two cases to consider:

\begin{figure}
    \centering 
    \def\len{1.3}
\begin{subfigure}{0.425\textwidth}
    \centering
    \begin{tikzpicture}
    \def\n{6}
     
    \def\r{0.5*\len cm / sin{(180/(2*\n))}}
    \def\R{0.5*\len cm / tan{(180/(2*\n))}}

    \foreach \k in {0,1,...,3} {
        \coordinate (p\k) at ({180/(2*\n)+(\k-1)*180/(\n)}:\r);
        \coordinate (q) at ({180/(2*\n)+\k*180/(\n)}:\r);
        
        \draw (p\k) -- (q);
    }

    \filldraw (p1) circle (1.5pt) node [below right] {$v_1$};
    \filldraw (p2) circle (1.5pt) node [right] {$v_2$};
    \filldraw (p3) circle (1pt) node [below left] {$u_2$};

    \coordinate (v) at (q);
    \filldraw (v) circle (1pt) node [below left] {$v$};
    
    \draw  ({-0.5*\R}, 0) -- ({1.9*\R}, 0) node[below] {$\mathcal{W}$};
    
    \coordinate (v3) at ($(p2)!1!120:(p1)$);
    \draw [ultra thick] (\R,0) -- (p1) -- (p2) -- (v3);
    \filldraw (v3) circle (1.5pt) node [below right] {$v_3$};

    \coordinate (tmp1) at (p2);
    \coordinate (tmp2) at (v3);
    \foreach \k in {0,1,...,3} {
        \coordinate (tmp3) at ($(tmp2)!1!120:(tmp1)$);
        \draw (tmp2) -- (tmp3);

        \coordinate (tmp1) at (tmp2);
        \coordinate (tmp2) at (tmp3);
    }
    \node [gray] (hex_center) at ($(tmp2)!1!60:(tmp1)$) {$P_p$};

    \coordinate (v4) at ($(v3)!1!-90:(p2)$);
    \draw [ultra thick] (v3)--(v4);
    \filldraw (v4) circle (1.5pt) node [label={[label distance=2pt]215:$v_4$}] {};

    \draw (v4)--(p3);
    \node [gray] (hex_center) at ($(v3)!0.7!45:(v4)$) {$P_\ell$};

    \coordinate (v5) at ($(v4)!1!-120:(p3)$);
    \draw [ultra thick] (v5)--(v4);
    \filldraw (v5) circle (1.5pt) node [below left] {$v_5$};

    \coordinate (h5) at ($(v5)!1!-120:(v4)$);
    \coordinate (h4) at ($(h5)!1!-120:(v5)$);
    
    \draw (v5) -- (h5) -- (h4) -- (v);
    \node [gray] (Pr_center) at ($(v5)!1!-60:(v4)$) {$P_r$};
    
    \coordinate (v6) at ($(v5)!1!125:(v4)$);
    \draw [ultra thick] (v6)--(v5);
    \filldraw (v6) circle (1.5pt) node [left] {$v_6$};

    \coordinate (u4) at ($(v4)!1!-125:(v5)$);
    \draw (v4)--(u4);
    \filldraw (u4) circle (1pt) node [right] {$u_4$};
    
    \coordinate (m5) at ($(v6)!0.5!(v5)$);
    \filldraw (m5) circle (1.5pt) node [label={[label distance=1pt]3:\small{$\pi_{\W_5}(v_5)$}}] {};
    \coordinate (W5a) at ($(m5)!2!-90:(v5)$);
    \coordinate (W5b) at ($(m5)!6!90:(v5)$);
    \draw (W5a) -- (W5b) node [above] {$\W_5$};
    
    \coordinate (pv4_W5) at ($(W5a)!(v4)!(W5b)$);
    \coordinate (pv4) at ($(-1,0)!(v4)!(1,0)$);
    \draw[dashed] (pv4) -- (v4) -- (pv4_W5) node [gray, right=8pt] {$P_s$};
    
    \coordinate (pv5) at ($(-1,0)!(v5)!(1,0)$);
    \draw[dashed] (v5)--(pv5);

    \path (pv4)--(v4)--(v3)
        pic ["{$\frac{\pi}{3}$}", draw, color=darkgray, angle radius=0.25*\len cm, angle eccentricity=1.7] {angle=pv4--v4--v3};

\path (pv4)--(v4)--(p3)
    pic ["{$\frac{\pi}{6}$}", draw, color=darkgray, angle radius=0.3*\len cm, angle eccentricity=1.7] {angle=p3--v4--pv4};

\path (p3)--(v4)--(v5)
    pic ["{$\frac{2\pi}{3}$}", draw, color=darkgray, angle radius=0.2*\len cm, angle eccentricity=1.65] {angle=v5--v4--p3};

\path (v5)--(v4)--(pv4_W5)
    pic ["{$\frac{\pi}{s}$}", draw, color=darkgray, angle radius=0.3*\len cm, angle eccentricity=1.6] {angle=pv4_W5--v4--v5};

\path (u4)--(v4)--(pv4_W5)
    pic ["{$\frac{s-2}{s}\pi$}", draw, color=darkgray, angle radius=0.25*\len cm, angle eccentricity=2.2] {angle=u4--v4--pv4_W5};

\path (u4)--(v4)--(v3)
    pic ["\small{$\geq\!\!\frac{\pi}{2}$}", draw, color=gray, angle radius=0.47*\len cm, angle eccentricity=1.6] {angle=v3--v4--u4};

\path (v4)--(v5)--(v6)
    pic ["\small{$<\!\!\frac{5\pi}{6}$}", draw, color=darkgray, angle radius=0.2*\len cm, angle eccentricity=2.2] {angle=v4--v5--v6};
        
    \coordinate (C) at (0,0) node [below, gray] {$P_q$};
\end{tikzpicture}
    \caption{}
    \label{fig:case_m3_W4_square_hexagon_no_v6}
\end{subfigure}
\begin{subfigure}{0.525\textwidth}
    \centering
    \begin{tikzpicture}
    \def\n{6}
        
    \def\r{0.5*\len cm / sin{(180/(2*\n))}}
    \def\R{0.5*\len cm / tan{(180/(2*\n))}}

    \foreach \k in {0,1,...,4} {
        \coordinate (p\k) at ({180/(2*\n)+(\k-1)*180/(\n)}:\r);
        \coordinate (q) at ({180/(2*\n)+\k*180/(\n)}:\r);
        
        \draw (p\k) -- (q);
    }

    \filldraw (p1) circle (1.5pt) node [below right] {$v_1$};
    \filldraw (p2) circle (1.5pt) node [right] {$v_2$};
    \filldraw (p3) circle (1pt) node [below] {$u_2$};

    \coordinate (v) at (p4);
    \coordinate (u) at (q);
    \filldraw (v) circle (1pt) node [above left=2pt and 5pt] {$v$};
    \filldraw (u) circle (1pt) node [below left] {$u$};
    
    \draw  ({-0.9*\R}, 0) -- ({1.9*\R}, 0) node[below] {$\mathcal{W}$};
    
    \coordinate (v3) at ($(p2)!1!120:(p1)$);
    \draw [ultra thick] (\R,0) -- (p1) -- (p2) -- (v3);
    \filldraw (v3) circle (1.5pt) node [below right] {$v_3$};

    \coordinate (tmp1) at (p2);
    \coordinate (tmp2) at (v3);
    \foreach \k in {0,1,...,3} {
        \coordinate (tmp3) at ($(tmp2)!1!120:(tmp1)$);
        \draw (tmp2) -- (tmp3);

        \coordinate (tmp1) at (tmp2);
        \coordinate (tmp2) at (tmp3);
    }
    \node [gray] (hex_center) at ($(tmp2)!1!60:(tmp1)$) {$P_p$};

    \coordinate (v4) at ($(v3)!1!-90:(p2)$);
    \draw [ultra thick] (v3)--(v4);
    \filldraw (v4) circle (1.5pt) node [above right] {$v_4$};

    \draw (v4)--(p3);
    \node [gray] (hex_center) at ($(v3)!0.7!45:(v4)$) {$P_\ell$};

    \coordinate (v5) at ($(v4)!1!-120:(p3)$);
    \draw [ultra thick] (v5)--(v4);
    \filldraw (v5) circle (1.5pt) node [above right] {$v_5$};

    \coordinate (m4) at ($(v5)!0.5!(v4)$);
    \coordinate (W4a) at ($(m4)!6!-90:(v4)$);
    \coordinate (W4b) at ($(m4)!2!90:(v4)$);
    \draw (W4a) -- (W4b) node [below=5pt] {$\W_4$};

    \coordinate (v6) at ($(v5)!1!-120:(v4)$);
    \coordinate (w) at ($(v6)!1!-120:(v5)$);
    \filldraw (w) circle (1pt) node[above left] {$w$};
    
    \draw (v5) -- (v6) -- (w) -- (v);
    \node [gray] (Pr_center) at ($(v6)!0.8!-60:(v5)$) {$P_r$};
    
    \draw [ultra thick] (v6)--(v5);
    \filldraw (v6) circle (1.5pt) node [above left] {$v_6$};

    \draw (w) -- ($(w)!1!-90:(v)$) -- (u);

    \coordinate (m) at ($(v)!0.5!(u)$);
    \filldraw (m) circle (1pt) node [left] {$m$};

    \coordinate (C) at (0,0);
    \filldraw (C) circle (1pt);
    \node [below, gray] (Pq_center) at ({0.5*\R}, 0) {$P_q$};
    
    \draw[dashed] ($(W4a)!(m)!(W4b)$) -- (m) -- (C);
    \draw [dashed] (v) -- (C);
    
    \path (v) -- (C) -- (m)
        pic ["{$\frac{\pi}{12}$}", draw, color=darkgray, angle radius=0.7*\len cm, angle eccentricity=1.4] 
        {angle=v--C--m};

    \path (v) -- (C) -- (\R, 0) coordinate (tmp)
        pic ["{$\frac{7\pi}{12}$}", draw, color=darkgray, angle radius=0.3*\len cm, angle eccentricity=1.7] 
        {angle=tmp--C--v};

    \path (-\R, 0) coordinate (tmp1) -- (C) -- (\R, 0) coordinate (tmp2)
        pic ["{$\frac{n}{6}\pi$}", draw, color=gray, angle radius=0.2*\len cm, angle eccentricity=1.7] 
        {angle=tmp1--C--tmp2};

    \path (-\R, 0) coordinate (tmp1) -- (C) -- (m)
        pic ["{$\frac{n-4}{6}\pi$}", draw, color=gray, angle radius=0.4*\len cm, angle eccentricity=1.7] 
        {angle=m--C--tmp1};

    \path (w) -- (v) -- (p3)
        pic ["{$\frac{2\pi}{3}$}", draw, color=darkgray, angle radius=0.15*\len cm, angle eccentricity=2] 
            {angle=p3--v--w};
    
    \tikzset{angleLine/.style={decoration={markings,
mark= at position 0.5 with
      with {
        \draw (0pt,-2pt) -- (0pt,2pt);
    } },
      pic actions/.append code=\tikzset{postaction=decorate}}}
      
    \path (C) -- (v) -- (p3)
        pic [angleLine, "{$\frac{5\pi}{12}$}", draw, color=darkgray, angle radius=0.25*\len cm, angle eccentricity=2]
            {angle=C--v--p3};

    \path (C) -- (v) -- (u)
        pic [angleLine, draw, color=darkgray, angle radius=0.25*\len cm]
            {angle=u--v--C};

    \path (w) -- (v) -- (u)
        pic [draw, color=gray, angle radius=0.15*\len cm]
            {right angle=u--v--w};

    \path (v) -- (m) -- ($(W4a)!(m)!(W4b)$) coordinate (tmp)
        pic [draw, color=gray, angle radius=0.125*\len cm]
            {right angle=v--m--tmp};

    \path (v) -- (m) -- (C)
        pic [draw, color=gray, angle radius=0.125*\len cm]
            {right angle=v--m--C};
\end{tikzpicture}
        \caption{$P_q$ has $2n$ sides, where $n>6$, so $\frac{n-4}{6}\pi\geq\frac{\pi}{2}$.}
        \label{fig:case_m3_W4_square_hexagon_v6}
\end{subfigure}
    
    \unskip
    \caption{Hexagon case (continued)}
    \label{fig:temp-fig-b+c_2}
\end{figure}

If $v_6$ is not a vertex of $P_r$, then $\angle_{v_5}(v_6, \pi_\W(v_5))=\frac{\pi}{6}\lplus \angle_{v_5}(v_4,v_6)$, and so $\W_5$ is disjoint from $\W$ unless $v_4, v_5, v_6$ bound a $2s$-truncated polygon $P_s$ with $s<6$ (\cref{fig:case_m3_W4_square_hexagon_no_v6}). Let $u_4\neq v_5$ be the other vertex of $P_s$ adjacent to $v_4$. By the link condition, we have $\angle_{v_4}(v_3, v_5)= \pi$ and $\angle_{v_4}(u_2, u_4)\geq \frac{\pi}{3}+\frac{\pi}{s}$. Furthermore, since $s<6$, we have $u_4\neq v_3$. Thus \cref{lem:angles_triangle_equality} gives 
\[
        \angle_{v_4}(\pi_\W(v_4), \pi_{\W_5}(v_4))=\min
        \left\{\begin{array}{l}
            \tfrac{\pi}{6}\lplus \tfrac{2\pi}{3}\lplus \tfrac{\pi}{s},\\
            \tfrac{\pi}{6} \lplus\angle_{v_4}(u_2,u_4)\lplus \frac{s-2}{s}\pi,  \\
            \tfrac{\pi}{3} \lplus\angle_{v_4}(v_3, v_5)\lplus \frac{\pi}{s}, \\
            \tfrac{\pi}{3} \lplus\angle_{v_4}(v_3,u_4)\lplus \frac{s-2}{s}\pi  \\
        \end{array}\right\} = \tfrac{\pi}{3} \lplus\angle_{v_4}(v_3,u_4)\lplus \tfrac{s-2}{s}\pi
\]

It follows that $\W_5$ is disjoint from $\W$, except possibly when $s=2$ and $u_4, v_3, v_4$ bound a square. However, after replacing $v_5$ with $u_4$ in $\gamma$, we see that this case has already been covered (\cref{fig:case_m3_W4_square}).

If $v_6$ is a vertex of $P_r$, then let $u\neq u_2$ be the other vertex of $P_q$ adjacent to $v$, and let $w\neq u_2$ be the other vertex of $P_r$ adjacent to $v$ (\cref{fig:case_m3_W4_square_hexagon_v6}). The wall dual to $[v,w]$ is $\W_4$, which we assumed intersects $\W$. The minimality of the length of $\gamma$ implies that $P_q$ has at least 14 sides. Because of this, $\pi_\W(v)$ is the center of $P_q$, and so~$\angle_v(u_2, \pi_\W(v))=\angle_v(u, \pi_\W(v))=\frac{5\pi}{12}$. Since $u_2, v, w$ bound a hexagon, we conclude from \cref{lem:bounding_2cells} that the vertices $u, v, w$ bound a square.
But then, if $m$ is the midpoint of the edge $[v, u]$, we also have that $\pi_\W(m)$ is the center of $P_q$, and so $\angle_m(\pi_\W(m), \pi_{\W_4}(m))=\pi$, contradicting the assumption that $\W_4$ intersects $\W$. 
\end{proof}

\begin{proposition}[Square case]\label[proposition]{prop:square_case}
    Suppose that $P_p$ is a square. Then $d_1(x, \W)\leq 4+\frac{1}{2}$.
\end{proposition}
\begin{proof}
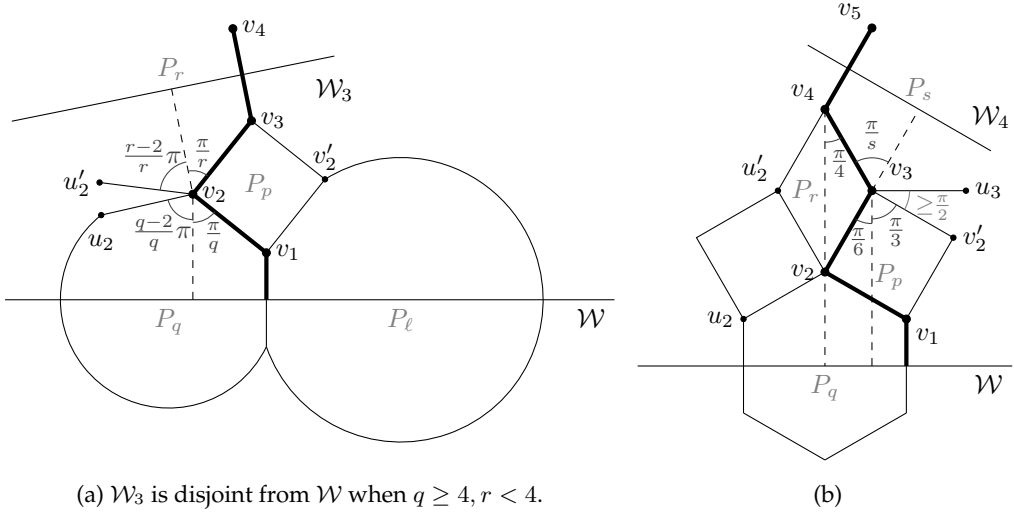
\begin{figure}
\centering
\centering
\def\len{1.25}
\begin{subfigure}{0.55\textwidth}
    \centering
    \begin{tikzpicture}
    \def\nq{3.5}
    \def\nl{((2*\nq)/(\nq-2))}

    \def\rq{0.5*\len cm / sin{(180/(2*\nq))}}
    \def\Rq{0.5*\len cm / tan{(180/(2*\nq))}}

    \def\rl{0.5*\len cm / sin{(180/(2*\nl))}}
    \def\Rl{0.5*\len cm / tan{(180/(2*\nl))}}

    \node [gray, below] (Cq) at (-\Rq, 0) {$P_q$};
    \foreach \k in {0,1,...,2} {
        \coordinate (v\k) at ($({180/(2*\nq)+(\k-1)*180/(\nq)}:\rq) - (\Rq, 0)$);
        \coordinate (q) at ($({180/(2*\nq)+\k*180/(\nq)}:\rq) - (\Rq, 0)$);
        
        \draw (v\k) -- (q);
    }
    \draw ($({180/(2*\nq)+2*180/(\nq)}:\rq)-(\Rq,0)$) arc ({180/(2*\nq)+2*180/(\nq)}:{360+180/(2*\nq)-180/(\nq)}:\rq);

    \coordinate (u2) at (q);
    \filldraw (u2) circle (1pt) node [below=4pt] {$u_2$};

    \coordinate (v2') at ($(v1)!1!{(\nl-1)*(180/\nl)}:(v0)$);
    \draw (v1)--(v2');
    
    \draw (v0) arc ({180-180/(2*\nl)+180/(\nl)-360}:{180-180/(2*\nl)-180/(\nl)}:\rl);

    \filldraw (v2') circle (1pt) node [above] {$v_2'$};
    \node [gray, below] (Cl) at (\Rl, 0) {$P_\ell$};

    \draw (-25-2*\Rq,0) -- (30+2*\Rl, 0) node [below left] {$\W$};

    \coordinate (v3) at ($(v2)!1!90:(v1)$);
    \coordinate (v4) at ($(v3)!1!-130:(v2)$);

    \filldraw (v1) circle (1.5pt) node [right] {$v_1$};
    \filldraw (v2) circle (1.5pt) node [right] {$v_2$};
    \filldraw (v3) circle (1.5pt) node [right] {$v_3$};
    \filldraw (v4) circle (1.5pt) node [right] {$v_4$};

    \draw[ultra thick] (v4) -- (v3) -- (v2) -- (v1) -- (0,0);
    
    \draw (v3) -- (v2');
    \node [gray] (Pp_center) at ($(v2)!0.5!(v2')$) {$P_p$};

    \coordinate (m3) at ($(v4)!0.5!(v3)$);
    \coordinate (W3a) at ($(m3)!5!-90:(v3)$);
    \coordinate (W3b) at ($(m3)!2!90:(v3)$);
    \draw (W3a) -- (W3b) node [below=5pt] {$\W_3$};

    \coordinate (u2') at ($(v2)!1!-20:(u2)$);
    \draw (v2) -- (u2');
    \filldraw (u2') circle (1pt) node [left] {$u_2'$};

    \coordinate (pv2) at ($(-1,0)!(v2)!(1,0)$);
    \coordinate (pv2_W3) at ($(W3a)!(v2)!(W3b)$);
    \draw [dashed] (pv2) -- (v2) -- (pv2_W3) node [gray, above] {$P_r$};

    \path (v1)--(v2)--(pv2)
        pic ["{$\frac{\pi}{q}$}", draw, color=darkgray, angle radius=0.3*\len cm, angle eccentricity=1.7] 
        {angle=pv2--v2--v1};

    \path (u2)--(v2)--(pv2)
        pic ["{$\frac{q-2}{q}\pi$}", draw, color=darkgray, angle radius=0.27*\len cm, angle eccentricity=1.95] 
        {angle=u2--v2--pv2};

    \path (u2')--(v2)--(pv2_W3)
        pic ["{$\frac{r-2}{r}\pi$}", draw, color=darkgray, angle radius=0.35*\len cm, angle eccentricity=1.7] 
        {angle=pv2_W3--v2--u2'};

    \path (v3)--(v2)--(pv2_W3)
        pic ["{$\frac{\pi}{r}$}", draw, color=darkgray, angle radius=0.25*\len cm, angle eccentricity=1.8] 
        {angle=v3--v2--pv2_W3};
    \end{tikzpicture}
    
\caption{$\W_3$ is disjoint from $\W$ when $q\geq 4, r<4$.}
\label{fig:claim_r<4}
\end{subfigure}
\begin{subfigure}{0.35\textwidth}
    \centering
    \begin{tikzpicture}
        \def\n{3}
            
        \def\r{0.5*\len cm / sin{(180/(2*\n))}}
        \def\R{0.5*\len cm / tan{(180/(2*\n))}}

        \foreach \k in {-3,-2,...,2} {
            \coordinate (v\k) at ({180/(2*\n)+(\k-1)*180/(\n)}:\r);
            \coordinate (q) at ({180/(2*\n)+\k*180/(\n)}:\r);
            
            \draw (v\k) -- (q);
        }
        \coordinate (Pq_center) at (0,0) node [gray, below] {$P_q$};

        \filldraw (v1) circle (1.5pt) node [below right] {$v_1$};
        \filldraw (v2) circle (1.5pt) node [left] {$v_2$};

        \coordinate (u2) at (q);
        \filldraw (u2) circle (1pt) node [left] {$u_2$};

        \draw  ({-2.3*\R}, 0) -- ({2.3*\R}, 0) node [below left] {$\mathcal{W}$};

        \coordinate (v3) at ($(v2)!1!90:(v1)$);
        \coordinate (v4) at ($(v3)!1!-120:(v2)$);
        \coordinate (v5) at ($(v4)!1!120:(v3)$);

        \filldraw (v3) circle (1.5pt) node [above right=1pt and 3pt] {$v_3$};
        \filldraw (v4) circle (1.5pt) node [above left] {$v_4$};
        \filldraw (v5) circle (1.5pt) node [above left] {$v_5$};

        \draw [ultra thick] (v5)--(v4)--(v3)--(v2)--(v1)--(\R,0);

        \coordinate (m4) at ($(v5)!0.5!(v4)$);
        \coordinate (W4a) at ($(m4)!1!-90:(v4)$);
        \coordinate (W4b) at ($(m4)!3.5!90:(v4)$);
        \draw (W4a) -- (W4b) node [above=5pt] {$\W_4$};

        \coordinate (v2') at ($(v3)!1!90:(v2)$);
        \filldraw (v2') circle (1pt) node [right] {$v_2'$};
        \draw (v3)--(v2')--(v1);

        \node [gray, below] (Pp_center) at ($(v2)!0.5!(v2')$) {$P_p$};
        
        \coordinate (u2') at ($(v2)!1!-90:(u2)$);
        \coordinate (tmp) at ($(u2')!1!-90:(v2)$);
        \draw (v2)--(u2')--(tmp)--(u2);
        \filldraw (u2') circle (1pt) node [above left] {$u_2'$};

        \draw (u2')--(v4);
        \node [gray, left] (Pr_center) at ($(u2')!0.55!(v3)$) {$P_r$};
        
        \coordinate (u3) at ($(v3)!1!-120:(v4)$);
        \draw (v3)--(u3);
        \filldraw (u3) circle (1pt) node [right] {$u_3$};

        \draw [dashed] (v4) -- (Pq_center);
        
        \coordinate (pv3) at ($(-1,0)!(v3)!(1,0)$);
        \coordinate (pv3_W4) at ($(W4a)!(v3)!(W4b)$);
        \draw [dashed] (pv3)--(v3)--(pv3_W4) node [gray, above] {$P_s$};

        \path (Pq_center)--(v4)--(v3)
            pic ["{$\frac{\pi}{4}$}", draw, color=darkgray, angle radius=0.35*\len cm, angle eccentricity=1.7] 
            {angle=Pq_center--v4--v3};

        \path (v2')--(v3)--(pv3)
            pic ["{$\frac{\pi}{3}$}", draw, color=darkgray, angle radius=0.3*\len cm, angle eccentricity=1.7] 
            {angle=pv3--v3--v2'};

        \path (v2)--(v3)--(pv3)
            pic ["{$\frac{\pi}{6}$}", draw, color=darkgray, angle radius=0.35*\len cm, angle eccentricity=1.6] 
            {angle=v2--v3--pv3};

        \path (v4)--(v3)--(pv3_W4)
            pic ["{$\frac{\pi}{s}$}", draw, color=darkgray, angle radius=0.35*\len cm, angle eccentricity=1.7] 
            {angle=pv3_W4--v3--v4};

        \path (u3)--(v3)--(v2')
            pic ["\small{$\geq\!\!\frac{\pi}{2}$}", draw, color=gray, angle radius=0.4*\len cm, angle eccentricity=1.7] 
            {angle=v2'--v3--u3};
    \end{tikzpicture}
    
    \caption{}
    \label{fig:case_m=2_c}
\end{subfigure}
\unskip
\caption{Square case}
\label{fig:case_m2}
\end{figure}
Let $v_2'$ be the vertex opposite to $v_2$ in $P_p$. Then the wall dual to $[v_1, v_2']$ is $\W_2$, which we have assumed intersects $\W$. Thus by definition of~$k$, the points $\pi_\W(v_1), v_1$ and $v_2'$ are contained in a $2\ell$-truncated polygon $P_\ell$ for some $\ell$. 
Up to replacing $v_2$ with $v_2'$ in $\gamma$, we may assume by \cref{lem:bounding_2cells} that $v_2,v_3,v_4$ bounds a $2r$-truncated polygon $P_r$ for some $r$ satisfying $\frac{\pi}{r}>\angle_{v_3}(v_2, \pi_\W(v_3))$. Note that $\angle_{v_3}(v_2, \pi_\W(v_3))= \frac{\pi}{6}$ when $q=3$ (since $q=3$ implies $\ell=6$) and $\angle_{v_3}(v_2, \pi_\W(v_3))\geq \frac{\pi}{4}$ when $q\geq 4$. In particular, we have~$\frac{\pi}{q}\lplus \frac{\pi}{2}\lplus \frac{\pi}{r}=\pi$. 

Let now $u_2\neq v_1$ be the other vertex of $P_q$ adjacent to $v_2$, and $u_2'\neq v_3$ be the other vertex in $P_r$ adjacent to $v_2$ (\cref{fig:claim_r<4}).  By the link condition, $\angle_{v_2}(v_1, u_2')\geq \frac{\pi}{2}+\frac{\pi}{r}$ and $\angle_{v_2}(u_1, v_3)\geq \frac{\pi}{2}+\frac{\pi}{r}$. 
Applying \cref{lem:angles_triangle_equality} first to the cell $P_q$, then to the cell $P_r$ gives
\[
        \angle_{v_2}(\pi_\W(v_2), \pi_{\W_3}(v_2))=\min
        \left\{\begin{array}{l}
            \tfrac{\pi}{q}\lplus \tfrac{\pi}{2}\lplus \tfrac{\pi}{r},\\
            \tfrac{\pi}{q} \lplus\angle_{v_2}(v_1,u_2')\lplus \frac{r-2}{r}\pi,  \\
            \tfrac{q-2}{q}\pi \lplus\angle_{v_2}(u_1, v_3)\lplus \frac{\pi}{r}, \\
            \tfrac{q-2}{q}\pi \lplus\angle_{v_2}(u_2,u_2')\lplus \frac{r-2}{r}\pi  \\
        \end{array}\right\} = \tfrac{q-2}{q}\pi \lplus\angle_{v_2}(u_2,u_2')\lplus \tfrac{r-2}{r}\pi
\]

\textbf{Claim:}\label{claim} \emph{
If $\W_3$ intersects $\W$, then $q=3$ and $r=2$.} 
Analogously, if $\W_3$ intersects $\W$ and $v_2', v_3, v_4$ bound a $2r'$-truncated polygon with $\frac{\pi}{r'}>\angle_{v_3}(v_2', \pi_\W(v_3))$, then $\ell=3$ and $r'=2$. 

\begin{proof}
Since $q>2$, it suffices to show that $\angle_{v_2}(\pi_\W(v_2), \pi_{\W_3}(v_2))=\pi$ whenever $q\geq 4$ or $r\geq 3$. 
If $u_2\neq u_2'$ then $\angle_{v_2}(u_2, u_2')\geq \frac{\pi}{2}$, and so 
$
        \angle_{v_2}(\pi_\W(v_2), \pi_{\W_3}(v_2))
        \geq \tfrac{q-2}{q}\pi\lplus\tfrac{\pi}{2}\lplus\tfrac{r-2}{r} = \pi
$
whenever $q\geq 4$ or $r\geq 3$.
If $u_2=u_2'$, then the link condition implies that $\frac{1}{q}+\frac{1}{r}<\frac{1}{2}$, and therefore $\tfrac{q-2}{q}\pi \lplus \tfrac{r-2}{r}\pi=2\pi-2\left(\frac{\pi}{q}+\frac{\pi}{r}\right)>\pi$. 
\end{proof}

We may now assume that $\W_3$ intersects $\W$, and thus that $P_q$ is a hexagon and $v_2, v_3, v_4$ bound a square. Furthermore, the link condition implies that $u_2\neq u_2'$ and since $\angle_{v_2}(\pi_\W(v_2), \pi_{\W_3}(v_2)) = \tfrac{\pi}{3}\lplus \angle_{v_2}(u_2, u_2')$, we get that $u_2, v_2, u_2'$ also bound a square.
In this case, $\pi_\W(v_4)$ is the center of $P_q$, and the geodesic $[v_4, \pi_\W(v_4)]$ contains $v_2$ (\cref{fig:case_m=2_c}). If $\W_4$ intersected $\W$, then by \cref{lem:bounding_2cells} we may assume, up to replacing $v_1$ with $u_2$ in $\gamma$, that $v_3, v_4, v_5$ bound a $2s$-truncated polygon $P_s$ with $s<4$ . Let $u_3\neq v_4$ be the other vertex of $P_s$ adjacent to $v_3$. By the link condition, $u_3\neq v_2'$, $\angle_{v_3}(v_2', v_4)=\pi$ and $\angle_{v_3}(u_3, v_2)\geq \frac{5\pi}{6}$. Thus, \cref{lem:angles_triangle_equality} gives 
\[
    \angle_{v_3}(\pi_\W(v_3), \pi_{\W_4}(v_3)) =\min
    \left\{\begin{array}{l}
        \tfrac{\pi}{6}\lplus\tfrac{\pi}{2}\lplus \tfrac{\pi}{s},\\
        \tfrac{\pi}{6}\lplus \angle_{v_3}(v_2, u_3)\lplus \tfrac{s-2}{s}\pi,\\
        \tfrac{\pi}{3}\lplus \angle_{v_3}(v_2', v_4)\lplus \tfrac{\pi}{s},\\
        \tfrac{\pi}{3}\lplus \angle_{v_3}(v_2', u_3)\lplus\tfrac{s-2}{s}\pi
    \end{array}\right\}
    = \tfrac{\pi}{3}\lplus \angle_{v_3}(v_2', u_3)\lplus\tfrac{s-2}{s}\pi
\]

In particular, $\W_4$ could only intersect $\W$ if $s=2$ and $v_2', v_3, c_4$ bound a square. But then after replacing $v_4$ with $u_3$ in $\gamma$, we get by the claim that $\ell=3$, contradicting the link condition. Therefore, $\W_4$ is disjoint from $\W$, concluding the proof.
\end{proof}

\bibliographystyle{alpha}
\bibliography{refs}

\begin{thebibliography}{MOP22}

\bibitem[BH93]{BrinkHowlett1993}
Brigitte Brink and Robert~B. Howlett.
\newblock A finiteness property and an automatic structure for {C}oxeter groups.
\newblock {\em Math. Ann.}, 296(1):179--190, 1993.

\bibitem[BH99]{BridsonHaefliger1999}
Martin~R. Bridson and Andr\'e Haefliger.
\newblock {\em Metric spaces of non-positive curvature}, volume 319 of {\em Grundlehren der mathematischen Wissenschaften [Fundamental Principles of Mathematical Sciences]}.
\newblock Springer-Verlag, Berlin, 1999.

\bibitem[Dav08]{Davis2008GeomCoxeterGps}
Michael~W. Davis.
\newblock {\em The geometry and topology of {C}oxeter groups}, volume~32 of {\em London Mathematical Society Monographs Series}.
\newblock Princeton University Press, Princeton, NJ, 2008.

\bibitem[HW14]{HruskaWise14}
G.~C. Hruska and Daniel~T. Wise.
\newblock Finiteness properties of cubulated groups.
\newblock {\em Compos. Math.}, 150(3):453--506, 2014.

\bibitem[JW13]{JanzenWise13}
David Janzen and Daniel~T. Wise.
\newblock Cubulating rhombus groups.
\newblock {\em Groups Geom. Dyn.}, 7(2):419--442, 2013.

\bibitem[MOP22]{Munro2021}
Zachary Munro, Damian Osajda, and Piotr Przytycki.
\newblock 2-dimensional {C}oxeter groups are biautomatic.
\newblock {\em Proc. Roy. Soc. Edinburgh Sect. A}, 152(2):382--401, 2022.

\bibitem[Mou88]{Moussong1988}
Gabor Moussong.
\newblock {\em Hyperbolic {C}oxeter groups}.
\newblock ProQuest LLC, Ann Arbor, MI, 1988.
\newblock Thesis (Ph.D.)--The Ohio State University.

\bibitem[NR03]{NibloReeves2003}
G.~A. Niblo and L.~D. Reeves.
\newblock Coxeter groups act on {${\rm CAT}(0)$} cube complexes.
\newblock {\em J. Group Theory}, 6(3):399--413, 2003.

\bibitem[OP25]{OsajdaPrzytycki}
Damian Osajda and Piotr Przytycki.
\newblock Coxeter groups are biautomatic.
\newblock {\em Invent. math.}, 2025.

\end{thebibliography}
\end{document}